\documentclass[a4paper,12pt]{article}
\usepackage{amssymb}
\usepackage{amsthm}
\usepackage{tikz,url}
\usepackage{tikz,pgfplots}
\usetikzlibrary{decorations.markings}
\usepackage{amsmath,amssymb}
\usepackage{todonotes}

\usepackage{color}
\usepackage{graphicx}
\usepackage{placeins,caption}

\DeclareMathOperator{\gp}{gp}
\DeclareMathOperator{\diam}{diam}

\DeclareMathOperator{\z}{z}

\DeclareMathOperator{\cp}{\,\square\,}

\DeclareMathOperator{\mono}{mp}
\DeclareMathOperator{\igp}{gp_i}
\DeclareMathOperator{\imp}{mp_i}
\DeclareMathOperator{\ext}{Ext}

\usepackage[top=3cm,bottom=3cm,left=2.8cm,right=2.8cm]{geometry}
\newtheorem{theorem}{Theorem}[section]
\newtheorem{lemma}[theorem]{Lemma}
\newtheorem{convention}[theorem]{Convention}
\newtheorem{corollary}[theorem]{Corollary}
\newtheorem{proposition}[theorem]{Proposition}

\newtheorem{conjecture}[theorem]{Conjecture}
\newtheorem{problem}[theorem]{Problem}

\theoremstyle{definition}

\newtheorem{definition}[theorem]{Definition}

\begin{document}
	
	\title{Colouring a graph with position sets}
	\author{Ullas Chandran S.V.$^{a}$ \\ \texttt{\footnotesize svuc.math@gmail.com}
		\and 
		Gabriele {Di Stefano}$^{b}$  \\ \texttt{\footnotesize gabriele.distefano@univaq.it}
		\and 
		Haritha S.$^{a}$\\ \texttt{\footnotesize harithasreelatha1994@gmail.com}
		\and
		Elias John Thomas$^{c,d}$ \\ \texttt{\footnotesize eliasjohnkalarickal@gmail.com}.
		\and 
		James Tuite$^{e,f}$\\ \texttt{\footnotesize james.t.tuite@open.ac.uk} 
	}

	\maketitle		
	{\centering\footnotesize The authors would like to dedicate this paper to the memory of the eminent Indian graph theorist Prof. E. Sampathkumar, who passed away 11th Aug 2024. He suggested the problem of general position colouring to Elias Thomas. We hope that he would have liked this paper.   \par}
	%\address{
		\noindent
		$^{a}$  Department of Mathematics, Mahatma Gandhi College, University of Kerala, \newline Thiruvananthapuram-695004, Kerala, India \\
		$^{b}$  Department of Information Engineering, Computer Science and Mathematics, \newline University of L'Aquila,  Italy\\
		$^{c}$  Department of Mathematics, Mar Ivanios College, University of Kerala, \newline Thiruvananthapuram-695015, Kerala, India \\
		$^{d}$  Department of Mathematics, Greenshaw High School, Grennell Road, Sutton, UK\\
		$^{e}$ School of Mathematics and Statistics, Open University, Milton Keynes, UK\\
		$^{f}$ Department of Informatics and Statistics, Klaip\.{e}da University, Lithuania
		%}
	
	\begin{abstract}
		In this paper we consider a colouring version of the general position problem. The \emph{$\gp $-chromatic number} is the smallest number of colours needed to colour the vertices of the graph such that each colour class has the no-three-in-line property. We determine bounds on this colouring number in terms of the diameter, general position number, size, chromatic number, cochromatic number and total domination number and prove realisation results. We also determine the $\gp $-chromatic number of several graph classes, including Kneser graphs $K(n,2)$, line graphs of complete graphs, complete multipartite graphs, block graphs and Cartesian products. Finally, we show that the $\gp $-colouring problem is NP-complete.
	\end{abstract}
	
	\noindent
	{\bf Keywords:}
	general position set, monophonic position set, graph colouring, Cartesian product, complexity

	\medskip\noindent
	{\bf AMS Subj.\ Class.\ (2020)}: 05C12, 05C15, 05C69
	
	%----------------------------------------------
	\section{Introduction}
	For any special property of a vertex subset of a graph $G$, it is interesting to consider partitions of the vertex set of $G$ such that each part of the partition has that property, or, equivalently, colourings of $V(G)$ such that each colour class has the given property. The most basic example is the property of being an independent set, in which case the associated colourings are \emph{proper} and the smallest possible number of colours in such a colouring is the \emph{chromatic number} of the graph. A dual concept is \emph{clique partitions}, which is equivalent to a proper colouring of the complement of the graph. Combining both concepts, a colouring is \emph{cochromatic} if each colour class induces either an independent set or a clique; see~\cite{Gim} for a discussion of this problem.
	
	Another example comes from domination theory; the \emph{domatic number} of $G$, introduced in~\cite{CH}, is the largest possible number of colours in a colouring of $V(G)$ such that each colour class is a dominating set. Other variations involve stronger distance conditions on the colour classes than independence; a \emph{2-distance colouring} requires two vertices in the same colour class to be at distance at least two (see~\cite{KraKra,KraKra2}), whilst \emph{injective colourings} are colourings in which each colour class is an \emph{open packing}, i.e. if $u,v$ are vertices with the same colour, then $u$ and $v$ have no common neighbours (see~\cite{BreSamYer,HahnKratSirSot}). 
	
	In this paper we consider a related colouring problem inspired by the \emph{general position problem} for graphs. This problem originated in a chessboard no-three-in-line puzzle of Dudeney~\cite{dudeney-1917}: how many pawns can be placed on an $n \times n$ chessboard such that no three pawns lie on a straight line in the plane? This question was generalised to graphs independently in~\cite{ullas-2016} and~\cite{manuel-2018a} as follows: given a graph $G$, what is the largest number of vertices in a subset $S \subseteq V(G)$ such that no three vertices in $S$ lie on a common shortest path of $G$? 
	
	Colourings and partitions related to no-three-in-line problems have been studied in the literature before. In~\cite{wood} Wood considers colourings of the $n \times n$ grid in which each colour class has Dudeney's geometrical no-three-in-line property. In finite geometry, a \emph{capset} is a set of points, no three of which lie on the same line; one problem of interest is to partition the vertices of a geometry into disjoint capsets, see for example~\cite{FolKalMcMPelWon}. A very recent paper~\cite{mvcoloring} treats a colouring problem for a related position parameter, the mutual-visibility number, which was introduced in~\cite{DiStefano}. 
	
	In this article we explore colouring problems for graphs in which each colour class is required to be in general position.	The plan of this paper is as follows. Section~\ref{sec:preliminary} defines basic graph theory terminology and poses the colouring problems that we will be investigating. Section~\ref{sec:bounds} provides some elementary bounds on our position colouring numbers and compares colouring numbers for different types of position sets, whilst Section~\ref{sec:graph parameters} explores connections with other well-known types of colourings and graph domination. In Section~\ref{sec:classes} we obtain the $\gp $-colouring numbers of some common graph classes that are of interest in the general position problem and Section~\ref{sec:small&large} discusses graphs with very large or small position colouring numbers. In Section~\ref{sec:cartesian} we investigate $\gp $-colourings of Cartesian products and strong grids. It is proven in Section~\ref{sec:complexity} that the decision version of the $\gp $-colouring problem is NP-complete. We close in Section~\ref{sec:conclusion} with some open problems.
	
	\section{Preliminaries}~\label{sec:preliminary}
	
	All graphs considered in this paper are finite, simple and undirected. We will denote the vertex set of a graph $G$ by $V(G)$ and the edge set by $E(G)$, and we will indicate that two vertices $u$ and $v$ are adjacent by writing $u \sim v$. The \emph{neighbourhood} $N(u)$ of a vertex $u$ is the set $\{ v \in V(G): u \sim v\} $. The \emph{degree} of a vertex $u$ is $\deg (u) = |N(u)|$. A vertex of degree one is a \emph{leaf} and a neighbour of a leaf is a \emph{support vertex}. A \emph{subgraph} of $G$ is a graph $H$ such that $V(H) \subseteq V(G)$ and $E(H) \subseteq E(G)$. When referring to neighbourhoods, degrees, etc.\ in a subgraph, we will add a subscript to make clear which graph is being referred to, e.g.\ $N_H(u)$ and $\deg _H(u)$. A subgraph $H$ of $G$ is \emph{induced} if for any pair $u,v \in V(H)$ we have $u \sim v$ in $H$ if and only if $u \sim v$ in $G$.
	
	We will write $[n] = \{ 1,2,\dots ,n\} $. A \emph{path} $P_n$ is the graph with vertex set $[n]$ such that $i \sim i+1$ for $1 \leq i \leq n-1$ and the \emph{length} of this path is $n-1$; a $u,v$-path in a graph $G$, $u,v \in V(G)$, is a subgraph isomorphic to a path with initial vertex $u$ and terminal vertex $v$. A \emph{cycle} $C_n$ of length $n$ has vertex set $\mathbb{Z}_n$ in which each vertex $i$ is adjacent to $i \pm 1$. The \emph{distance} $d(u,v)$ from $u$ to $v$ is the length of a shortest $u,v$-path. The \emph{diameter} $\diam (G)$ of $G$ is $\max \{ d(u,v):u,v \in V(G)\} $, and the \emph{monophonic diameter} $\diam _m(G)$ is the length of a longest induced path in $G$.  A subgraph $H$ is an \emph{isometric} subgraph of $G$ if $d_H(u,v) = d_G(u,v)$ for all $u,v \in V(H)$.
	
	An \emph{independent set} of $G$ is a set of mutually non-adjacent vertices and a \emph{clique} is a set of mutually adjacent vertices. The number of vertices in a largest independent set and a largest clique are respectively the \emph{independence number} $\alpha (G)$ and the \emph{clique number} $\omega (G)$ of $G$. An \emph{independent union of cliques} in $G$ is a disjoint union of cliques with no edges between distinct cliques in the collection. A vertex is \emph{extreme} if its neighbourhood induces a clique, and the set of all extreme vertices of $G$ will be denoted by $\ext (G)$. The \emph{complement} $\overline{G}$ of a graph $G$ is the graph with vertex set $V(\overline{G}) = V(G)$ in which a pair of vertices $u,v$ is adjacent if and only if they are non-adjacent in $G$. We write the disjoint union of two graphs $G,H$ as $G \dot \cup H$ (and when we write `$G$ is disjoint union of cliques', we allow the possibility that $G$ is a clique). The \emph{join} $G \vee H$ of two graphs $G,H$ is the graph formed from the disjoint union $G \dot \cup H$ by adding all edges between $G$ and $H$.
	
	A \emph{colouring} of $G$ with $k$ colours is a function $\rho :V(G) \rightarrow [k]$, although we will also refer to the colours as red, blue, etc. A colouring is \emph{proper} if no two adjacent vertices are assigned the same colour. The \emph{chromatic number} $\chi (G)$ is the smallest number of colours in a proper colouring of $G$. A \emph{clique covering} is a colouring such that no pair of non-adjacent vertices receives the same colour and the smallest number of colours in a clique covering is the \emph{clique cover number} $\theta (G)$. Observe that $\chi (G) = \theta (\overline{G})$. A \emph{cocolouring} is a colouring of $G$ in which each colour class induces either a clique or an independent set, and the smallest number of colours needed for a cocolouring is the \emph{cochromatic number} $\z (G)$. More information on cocolourings can be found in~\cite{Gim}. 
	
	The \emph{total domination number} $\gamma _t (G)$ is the number of vertices in a smallest subset $S \subset V(G)$ such that every vertex of $G$ is a neighbour of some vertex in $S$. A graph is an \emph{efficient open domination graph} if it contains a subset $D \subseteq V(G)$ such that $V(G) = \bigcup _{u \in D}N(u)$, but $N(u) \cap N(v) = \emptyset $ when $u,v \in D$ and $u \not = v$. A clique of order $n$ with one edge deleted is written $K_n^-$. A \emph{diamond} is copy of $K_4^-$ and we will call a graph $G$ \emph{diamond-free} if it contains no subgraph isomorphic to a diamond. The \emph{Cartesian product} $G \cp H$ of two graphs $G,H$ is the graph with vertex set $V(G) \times V(H)$ in which a vertex $(u,v)$ is adjacent to $(u',v')$ if and only if either i) $u = u'$ and $v \sim v'$ in $H$ or ii) $u \sim u'$ in $G$ and $v = v'$. The \emph{strong product} $G \boxtimes H$ of graphs $G$ and $H$ has vertex set $V(G) \times V(H)$ and has an edge between $(u,v)$ and $(u',v')$ if i) $u = u'$ and $v \sim v'$ in $H$, ii) $u \sim u'$ in $G$ and $v = v'$, or iii) $u \sim u'$ in $G$ and $v \sim v'$ in $H$.

	We now define the vertex subsets that will form the focus of this paper.
	
	\begin{definition}\label{def:gp}
		A subset $S \subseteq V(G)$ is in \emph{general position}, or is a \emph{general position set}, if no shortest path of $G$ contains more than two vertices of $S$. The \emph{general position number} $\gp (G)$ is the number of vertices in a largest general position set of $G$. 
	\end{definition}
	
	Several variants of the general position number have been considered in the literature (see~\cite{survey} for a survey), some of which will be needed in our study. The \emph{monophonic position problem} is defined by replacing `shortest path' in Definition~\ref{def:gp} by `induced path' and was first investigated in~\cite{ThoChaTui}. A set $M$ of vertices of a graph $G$ is in \emph{monophonic position} or is a \emph{monophonic position set} if no induced path of $G$ contains more than two vertices of $M$. The \emph{monophonic position number} $\mono (G)$ (or \emph{$\mono $-number} for short) is the number of vertices in a largest monophonic position set of $G$. 
	
	The \emph{mutual-visibility problem}, defined in~\cite{DiStefano}, is a relaxation of the no-three-in-line property: a set $S \subseteq V(G)$ is \emph{mutually-visible} if for any pair $u,v \in S$ there exists at least one shortest $u,v$-path in $G$ that does not pass through a vertex of $S - \{ u,v\} $. The largest number of vertices in a mutual-visibility set is the \emph{mutual-visibility number} $\mu (G)$. 
	
	An \emph{independent general position set} is a general position set that is also an independent set. Such sets were considered in~\cite{ThoCha}. The \emph{independent general position number} is the number of vertices in a largest independent general position set and we will denote this by $\igp (G)$. Similarly, an \emph{independent monophonic position set} is a monophonic position set that is also an independent set, and the number of vertices in a largest independent monophonic position set of $G$ is the \emph{independent monophonic position number} $\imp (G)$.
	
	Since we will frequently wish to discuss several types of position set at the same time, we adopt the following convention.
	
	\begin{convention}
		If $\pi $ stands for a position type invariant (in this context $\mu (G)$, $\gp (G) $ or $\mono (G)$), then we call a subset of that type a \emph{$\pi $-set}. The corresponding independent $\pi $-number will be denoted by $\pi _i(G)$.
	\end{convention}
	
	It is easy to see that for any connected graph we have $\mono (G) \leq \gp (G) \leq \mu (G)$. The graphs that we consider will not necessarily be connected; in this case, for any of the position type invariants $\pi $ described above, we will adopt the convention that a vertex subset $S$ of a disconnected graph $G$ is a $\pi $-set of $G$ if and only if $S \cap V(H)$ is a $\pi $-set of $H$ for each component $H$ of $G$.  
	
	We may now formally define the colouring problems introduced in this article. 
	
	\begin{definition}\label{def:main definition}
		Let $\pi (G)$ be one of $\gp (G), \mono (G)$ or $\mu (G)$. A \emph{$\pi $-colouring} of a graph $G$ is a colouring of $G$ such that each colour class is a $\pi $-set. The \emph{$\pi $-chromatic number} $\chi _{\pi }(G)$ is the smallest number of colours needed for a $\pi $-colouring of $G$. For independent position sets we define $\pi _i$-colourings and $\chi _{\pi _i}$-chromatic numbers analogously.
	\end{definition}
	
	In particular, the \emph{$\gp $-chromatic number} $\chi _{\gp }(G)$ is the smallest number of colours needed to colour $V(G)$ such that each colour class is a general position set of $G$. Observe that for disconnected graphs, the $\pi $-chromatic number will be the maximum of the $\pi $-chromatic numbers of the components. An example of these concepts can be seen in Figure~\ref{fig:Petersen}, which displays optimal $\gp $-, $\mu $- and $\mono $-colourings of the Petersen graph. A pleasing fact is that the colour classes on the left of Figure~\ref{fig:Petersen} are a largest general position set (red) and lower general position set (blue), i.e. a smallest maximal general position set, see~\cite{lowergp}.
	
	\begin{figure}
		\centering
		\begin{tikzpicture}[x=0.2mm,y=-0.2mm,inner sep=0.2mm,scale=0.9,thick,vertex/.style={circle,draw,minimum size=13,fill=lightgray}]
			\node at (-280,200) [vertex,color=red] (v1) {};
			\node at (-108.8,324.4) [vertex,color=blue] (v2) {};
			\node at (-174.2,525.6) [vertex,color=red] (v3) {};
			\node at (-385.8,525.6) [vertex,color=red] (v4) {};
			\node at (-451.2,324.4) [vertex,color=blue] (v5) {};
			\node at (-280,272) [vertex,color=red] (v6) {};
			\node at (-216.5,467.4) [vertex,color=blue] (v7) {};
			\node at (-382.7,346.6) [vertex,color=red] (v8) {};
			\node at (-177.3,346.6) [vertex,color=red] (v9) {};
			\node at (-343.5,467.4) [vertex,color=blue] (v10) {};

			\node at (180,200) [vertex,color=red] (u1) {};
			\node at (8.8,324.4) [vertex,color=blue] (u2) {};
			\node at (74.2,525.6) [vertex,color=pink] (u3) {};
			\node at (285.8,525.6) [vertex,color=pink] (u4) {};
			\node at (351.2,324.4) [vertex,color=blue] (u5) {};
			\node at (180,272) [vertex,color=blue] (u6) {};
			\node at (116.5,467.4) [vertex,color=red] (u7) {};
			\node at (282.7,346.6) [vertex,color=green] (u8) {};
			\node at (77.3,346.6) [vertex,color=green] (u9) {};
			\node at (243.5,467.4) [vertex,color=red] (u10) {};

			\path
			(v1) edge (v2)
			(v1) edge (v5)
			(v2) edge (v3)
			(v3) edge (v4)
			(v4) edge (v5)
			(v6) edge (v7)
			(v6) edge (v10)
			(v7) edge (v8)
			(v8) edge (v9)
			(v9) edge (v10)
			
			(v2) edge (v9)
			(v1) edge (v6)
			(v3) edge (v7)
			(v4) edge (v10)
			(v5) edge (v8)

			(u1) edge (u2)
			(u1) edge (u5)
			(u2) edge (u3)
			(u3) edge (u4)
			(u4) edge (u5)
			(u6) edge (u7)
			(u6) edge (u10)
			(u7) edge (u8)
			(u8) edge (u9)
			(u9) edge (u10)
			
			(u2) edge (u9)
			(u1) edge (u6)
			(u3) edge (u7)
			(u4) edge (u10)
			(u5) edge (u8)
			;
		\end{tikzpicture}
		\caption{An optimal $\gp $- and $\mu $-colouring (left) and an optimal $\mono $-colouring (right) of the Petersen graph}
		\label{fig:Petersen}
	\end{figure}
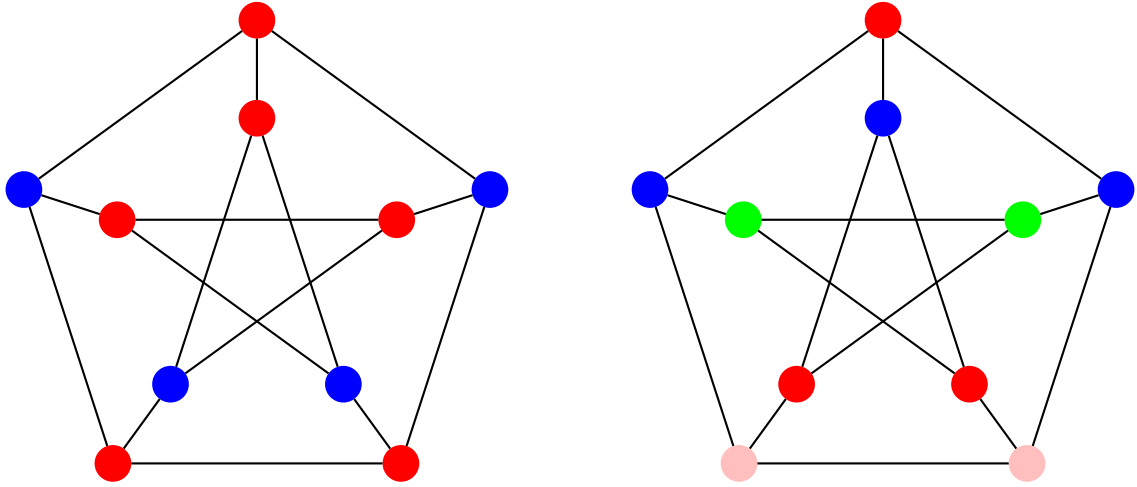

	%%%%%%%%%%%%%%%%%%%%%%
	%%%%%%%%%%%%%%%%%%%%%%
	\section{Elementary bounds on position colouring numbers}\label{sec:bounds}
	%%%%%%%%%%%%%%%%%%%%%%
	%%%%%%%%%%%%%%%%%%%%%%
	In this section we discuss some elementary bounds on position-type colourings. First we note an elementary connection between some of the position chromatic numbers. As any monophonic position set is also a general position set, and any general position set is also a mutual-visibility set, we immediately obtain the following inequality.
	
	\begin{lemma}\label{lem:chi_xi inequality}
		For any graph $G$, \[ \chi _{\mu }(G) \leq \chi _{\gp }(G) \leq \chi _{\mono }(G)\] and \[ \chi _{\mu _i}(G) \leq \chi _{\gp _i}(G) \leq \chi _{\mono _i}(G).\]  When $\pi (G)$ is any of $\gp (G)$, $\mu (G)$ or $\mono (G)$, $\chi _{\pi }(G) \leq \chi _{\pi _i}(G)$.
	\end{lemma}
	
	For any position-type invariant $\pi $ we have the following trivial bound. For mutual-visibility colourings, this result can be found in Lemma 2.1 and Proposition 5.3 of~\cite{mvcoloring}.
	
	\begin{lemma}\label{lem:ni/xi bound}
		If $G$ is a graph with order $n$ and $\pi $ is any position-type invariant, then \[ \left \lceil \frac{n}{\pi (G)}\right \rceil \leq \chi _{\pi }(G) \leq n-\pi (G)+1.\] If $\pi (G)$ is one of $\mu (G)$, $\gp (G)$ or $\mono (G)$, then the upper bound can be improved to
		\[ \chi _{\pi }(G) \leq \left \lceil \frac{n-\pi(G)+2}{2} \right \rceil .\]
	\end{lemma}
	\begin{proof}
		Let $\rho $ be a partition of $G$ into $\pi $-sets. Each set in $\rho $ contains at most $\pi (G)$ vertices, so that $\rho $ must have at least $\frac{n}{\pi (G)}$ parts; this establishes the lower bound.
		
		For the upper bound, let $S$ be a largest $\pi $-set of $G$. Create a colouring of $G$ by colouring the vertices of $S$ red, then giving all other vertices different colours. If $\pi (G)$ is one of $\mu (G)$, $\mono (G)$ or $\gp (G)$, then any set of two vertices is in $\pi $-position, so that we can colour the remaining vertices in pairs (with the possible exception of one left-over vertex). This establishes the upper bound.
	\end{proof}
	
	Observe that any graph with $\gp $- or $\mono $-number two will automatically meet the lower bound in Lemma~\ref{lem:ni/xi bound}, since any pair of vertices is in general and monophonic position. The only graphs with general position number two are $C_4$ and paths with order at least two. By contrast, it is shown in~\cite{ThoChaTui} and~\cite{extremal position} that there is a wide variety of graphs with monophonic position number two.
	
	\begin{corollary}\label{cor:paths}
		If $G$ is a path $P_n$, $n \geq 1$, or $C_4$, then $\chi _{\gp }(G) = \left \lceil \frac{n}{2} \right \rceil $. If $G$ has monophonic position number two, then $\chi _{\mono }(G) = \left \lceil \frac{n}{2} \right \rceil$.
	\end{corollary}
	In particular, it follows that for cycles of length at least four, we have $\chi _{\mono }(C_n) = \left \lceil \frac{n}{2} \right \rceil $. The $\gp $-colourings of cycles with length at least five also follows easily from Lemma~\ref{lem:ni/xi bound}.
	
	\begin{corollary}
		For $n \geq 5$, we have $\chi _{\gp }(C_n) = \left \lceil \frac{n}{3} \right \rceil $.
	\end{corollary}
	\begin{proof}
		Recall that we identify the vertices of the cycle $C_n$ with $\mathbb{Z}_n$. For $n = 5$, we can colour the vertices $0,2,3$ red and the vertices $1,4$ blue, so we may assume that $n\geq 6$. Write $n = 3q+r$, where $0 \leq r \leq 2$ and $q \geq 2$. For $0 \leq i \leq q-1$, let $S_i = \{ i,q+i,2q+i\} $. Each of these sets is in general position and there are $r$ vertices left over that form a general position set. Hence we have a $\gp $-colouring of $C_n$ with $q$ colours when $3|n$ and $q+1$ colours otherwise.
	\end{proof}

	Lemma~\ref{lem:ni/xi bound} immediately leads to a characterisation of graphs with very large or small position colouring numbers.
	
	\begin{lemma}\label{lem:cliques} \text{If $\pi (G)$ is $\gp (G)$, $\mu (G)$ or $\mono (G)$, then }
		\begin{itemize}
			\item $\chi _{\pi }(G) = 1$ if and only if $G$ is a disjoint union of cliques.
			\item $\chi _{\pi _i}(G) = 1$ if and only if $G$ is an empty graph, and $\chi _{\pi _i}(G) = n$ if and only if $G$ is a clique.
		\end{itemize}
	\end{lemma}

	\begin{lemma}\label{lem:isometric}
		If $H$ is an isometric subgraph of $G$, then $\chi _{\gp }(G) \geq \chi _{\gp }(H)$, and if $H$ is an induced subgraph of $G$, then $\chi _{\mono }(G) \geq \chi _{\mono }(H)$.
	\end{lemma}
	\begin{proof}
		Suppose that $H$ is an isometric subgraph of $G$. If $S$ is any general position set of $G$, then $S \cap V(H)$ is in general position in $H$. It follows, upon taking the intersection of each colour class of an optimal $\gp $-colouring of $G$ with $V(H)$ that $H$ has a $\gp $-colouring with at most $\chi _{\gp }(G)$ colours. The reasoning for $\mono $-colourings and induced subgraphs is similar.
	\end{proof}
	
	Applying Lemma~\ref{lem:isometric} and using Corollary~\ref{cor:paths} gives the following bounds in terms of the diameter and monophonic diameter.
	
	\begin{corollary}\label{lem:diameter bound} 
		For any graph $G$ \[ \chi _{\gp}(G) \geq \left \lceil \frac{\diam^*(G)+1}{2} \right \rceil ,\] where $\diam ^*(G) = \max \{ \diam (W) : W \text{ is a component of } G\} $. The same inequalities hold for $\chi _{\mono }$ if the diameter is replaced by the monophonic diameter $\diam _m(G)$.
	\end{corollary}
	
	It turns out that Corollary~\ref{lem:diameter bound} is tight for the class of block graphs.
	
	\begin{theorem}\label{thm:block graph} 
		For a block graph $G$, $\chi _{\gp}(G)= \left \lceil \frac{\diam(G)+1}{2} \right \rceil $.
	\end{theorem}
	\begin{proof}
		Set $s = \left \lceil \frac{\diam(G)+1}{2} \right \rceil $. The lower bound follows from Corollary~\ref{lem:diameter bound}, so we need only produce a $\gp $-colouring of $G$ with $s$ colours. Set $G_0 = G$ and for $1 \leq i \leq s$ we define $V_i = \ext (G_{i-1})$ and $G_i = G_{i-1}-V_i$. Note that $V_s = V(G_{s-1})$, $V(G) = \bigcup _{i=1}^sV_i$ and $G_i$ is an isometric subgraph of $G$ for $0 \leq i \leq s-1$. For $1 \leq i \leq s$, the subset $V_i$ is a general position set of $G_{i-1}$ and hence is in general position in $G$. 
	\end{proof}
	
	Theorem~\ref{thm:block graph} allows us to prove the existence of a graph of order $n$ and $\gp $- or $\mono $-chromatic number $a$ for $n \geq 2a-1$.
	
	\begin{theorem}\label{thm:existence}
		If $\pi (G)$ is $\gp (G)$ or $\mono (G)$, then for $a \geq 1$ there exists a graph $G$ with order $n$ and $\chi _{\pi }(G) = a$ if and only if $n \geq 2a-1$.
	\end{theorem}
	\begin{proof}
		Let $G$ be a graph with $\chi _{\gp }(G) = a$. Then by the upper bound in Lemma~\ref{lem:ni/xi bound} we have $n \geq 2a-1$. If $a = 1$, then a clique of order $n$ will suffice, so suppose that $a \geq 2$. Let $T$ be the tree formed by adding $n-2a+1$ leaves to the vertex $2a-2$ of the path $P_{2a-1}$. $T$ has order $n$ and by Theorem~\ref{thm:block graph} has $\gp $-chromatic number $\chi _{\gp }(T) = a$. As all paths in $T$ are unique, $\chi _{\mono }(T) = \chi _{\gp }(T)$.
	\end{proof}
	
	We now expand on Lemma~\ref{lem:chi_xi inequality} by proving realisation results for $\chi _{\mu }$, $\chi _{\gp }$ and $\chi _{\mono }$.
	\begin{theorem}
		There exists a graph with $\chi _{\gp}(G) = a$ and $\chi _{\mono }(G) = b$ if and only if $a = b = 1$ or $2 \leq a \leq b$.
	\end{theorem}
	\begin{proof}
		Suppose that there is a graph $G$ with $\chi _{\gp}(G) = a$ and $\chi _{\mono }(G) = b$. By Lemma~\ref{lem:chi_xi inequality} we must have $a \leq b$. If $a = 1$, then by Lemma~\ref{lem:cliques} $G$ is a clique and we must also have $b = 1$, so assume that $a \geq 2$. By Corollary~\ref{cor:paths} we have $\chi _{\gp }(P_{2a-1}) = \chi_{\mono}(P_{2a-1}) = a$, so we can further assume that $b > a$.
		
		For $r \geq 3$, form the crown graph $K(r)$ from the complete bipartite graph $K_{r,r}$ with partite sets $X = \{x_1,\dots,x_r\} $ and $Y = \{ y_1,\dots,y_r\} $ by deleting the edge $x_iy_i$ for $1 \leq i\ \leq r$. Now for $s \geq 1$ define $H(r,s)$ to be the graph with order $2r+s$ formed from the disjoint union of $K(r)$ with a path $P_s$ with vertex set $[s]$ by adding an edge from the vertex $1$ of $P$ to every vertex of $X$. An example is shown in Figure~\ref{fig:H(r,s) gp vs mp gp-colouring}. 
		
		The sets $Y \cup \{ 1\}$ and $X \cup \{ 2\} $ (or just $X$, if $s = 1$) are both in general position, and so can be coloured with two colours. If $s \geq 3$, the remaining $s-2$ vertices of the path $P_s$ can be coloured with $\left \lceil \frac{s-2}{2} \right \rceil$ further colours, so $\chi _{\gp }(H(r,s)) \leq \left \lceil \frac{s+2}{2} \right \rceil $. The diameter of $H(r,s)$ is $\max \{ 3,s+1\} $, so it follows from Corollary~\ref{lem:diameter bound} that $\left \lceil \frac{s+2}{2} \right \rceil $ is also a lower bound. Hence $\chi _{\gp }(H(r,s)) = \left \lceil \frac{s+2}{2} \right \rceil $ for $r \geq 3, s \geq 1$.
		
		Let $M$ be a monophonic position set of $H(r,s)$. $M$ can contain at most two vertices of the path $P_s$ and, if it does contain two such vertices, then $|M| = 2$. It follows from Theorem 3.5 of~\cite{extremal position} that the subgraph of $H(r,s)$ induced by $V(K(r)) \cup \{ 1\}$ has monophonic position number two. Therefore $|M| \leq 3$ and we have $|M| = 3$ if and only if $M$ contains one vertex of $P_s - \{ 1\} $ and i) two vertices of $X$, ii) two vertices of $Y$, or iii) a pair $\{ x_i,y_i\} $ for some $1 \leq i \leq r$. Thus if $r \leq s-1$, then \[ \chi _{\mono }(H(r,s)) = r+\left \lceil \frac{s-r}{2} \right \rceil = \left \lceil \frac{r+s}{2} \right \rceil .\] If $r \geq s$, then only $s-1$ colour classes can contain three vertices of $H(r,s)$, so that \[ \chi _{\mono }(H(r,s)) = s-1 + \left \lceil \frac{2r+s-3(s-1)}{2} \right \rceil = r+1. \] For fixed $s \geq 1$, if we vary $r$ from $3$ to $s-1$, then we get all values of $\chi _{\mono }$ from $\left \lceil \frac{s+3}{2} \right \rceil $ to $s$, and for $r \geq s$ we get all values of $\chi _{\mono }$ from $s+1$ onwards. Hence, setting $s = 2a-1$ and increasing $r$ from $3$ yields all the required values for $b > a \geq 2$. This completes our proof.
		
		\begin{figure}\centering
			\begin{tikzpicture}[x=0.2mm,y=-0.2mm,inner sep=0.1mm,scale=1.50,
				thick,vertex/.style={circle,draw,minimum size=10,font=\tiny,fill=white},edge label/.style={fill=white}]
				\tiny
				\node at (-350,-25) [vertex,color=blue] (x1) {$x_1$};
				\node at (-290,-25) [vertex,color=blue] (x2) {$x_2$};
				\node at (-230,-25) [vertex,color=blue] (x3) {$x_3$};
				\node at (-170,-25) [vertex,color=blue] (x4) {$x_4$};
				\node at (-110,-25) [vertex,color = blue] (x5) {$x_5$};
				
				\node at (-350,25) [vertex,color=red] (y1) {$y_1$};
				\node at (-290,25) [vertex,color=red] (y2) {$y_2$};
				\node at (-230,25) [vertex,color=red] (y3) {$y_3$};
				\node at (-170,25) [vertex,color=red] (y4) {$y_4$};
				\node at (-110,25) [vertex,color=red] (y5) {$y_5$};
				
				\node at (-230,-50) [vertex,color=red] (z0) {$z_0$};
				\node at (-230,-75) [vertex,color=blue] (z1) {$z_1$};
				\node at (-230,-100) [vertex,color=yellow] (z2) {$z_2$};
				\node at (-230,-125) [vertex,color=yellow] (z3) {$z_3$};
				\node at (-230,-150) [vertex,color=green] (z4) {$z_4$};
				\node at (-230,-175) [vertex,color=green] (z5) {$z_5$};

				\node at (-50,-25) [vertex,color=red] (a1) {$x_1$};
				\node at (10,-25) [vertex,color=blue] (a2) {$x_2$};
				\node at (70,-25) [vertex,color=blue] (a3) {$x_3$};
				\node at (130,-25) [vertex,color=yellow] (a4) {$x_4$};
				\node at (190,-25) [vertex,color=yellow] (a5) {$x_5$};
				
				\node at (-50,25) [vertex,color=red] (b1) {$y_1$};
				\node at (10,25) [vertex,color=green] (b2) {$y_2$};
				\node at (70,25) [vertex,color=green] (b3) {$y_3$};
				\node at (130,25) [vertex,color=pink] (b4) {$y_4$};
				\node at (190,25) [vertex,color=pink] (b5) {$y_5$};
				
				\node at (70,-50) [vertex,color=brown] (c0) {$z_0$};
				\node at (70,-75) [vertex,color=pink] (c1) {$z_1$};
				\node at (70,-100) [vertex,color=green] (c2) {$z_2$};
				\node at (70,-125) [vertex,color=yellow] (c3) {$z_3$};
				\node at (70,-150) [vertex,color=blue] (c4) {$z_4$};
				\node at (70,-175) [vertex,color=red] (c5) {$z_5$};
				
				\path
				
				(z0) edge (x1)
				(z0) edge (x2)
				(z0) edge (x3)
				(z0) edge (x4)
				(z0) edge (x5)
				
				(z0) edge (z1)
				(z1) edge (z2)
				(z2) edge (z3)
				(z3) edge (z4)
				(z4) edge (z5)
				
				(x1) edge (y2)
				(x1) edge (y3)
				(x1) edge (y4)
				(x1) edge (y5)
				
				(x2) edge (y1)
				(x2) edge (y3)
				(x2) edge (y4)
				(x2) edge (y5)
				
				(x3) edge (y2)
				(x3) edge (y1)
				(x3) edge (y4)
				(x3) edge (y5)
				
				(x4) edge (y2)
				(x4) edge (y3)
				(x4) edge (y1)
				(x4) edge (y5)
				
				(x5) edge (y2)
				(x5) edge (y3)
				(x5) edge (y4)
				(x5) edge (y1)

				(c0) edge (a1)
				(c0) edge (a2)
				(c0) edge (a3)
				(c0) edge (a4)
				(c0) edge (a5)
				
				(c0) edge (c1)
				(c1) edge (c2)
				(c2) edge (c3)
				(c3) edge (c4)
				(c4) edge (c5)
				
				(a1) edge (b2)
				(a1) edge (b3)
				(a1) edge (b4)
				(a1) edge (b5)
				
				(a2) edge (b1)
				(a2) edge (b3)
				(a2) edge (b4)
				(a2) edge (b5)
				
				(a3) edge (b2)
				(a3) edge (b1)
				(a3) edge (b4)
				(a3) edge (b5)
				
				(a4) edge (b2)
				(a4) edge (b3)
				(a4) edge (b1)
				(a4) edge (b5)
				
				(a5) edge (b2)
				(a5) edge (b3)
				(a5) edge (b4)
				(a5) edge (b1)
				
				;

			\end{tikzpicture}
			\caption{$H(5,6)$ with an optimal $\gp $-colouring (left) and $\mono $-colouring (right).}
			\label{fig:H(r,s) gp vs mp gp-colouring}
		\end{figure}
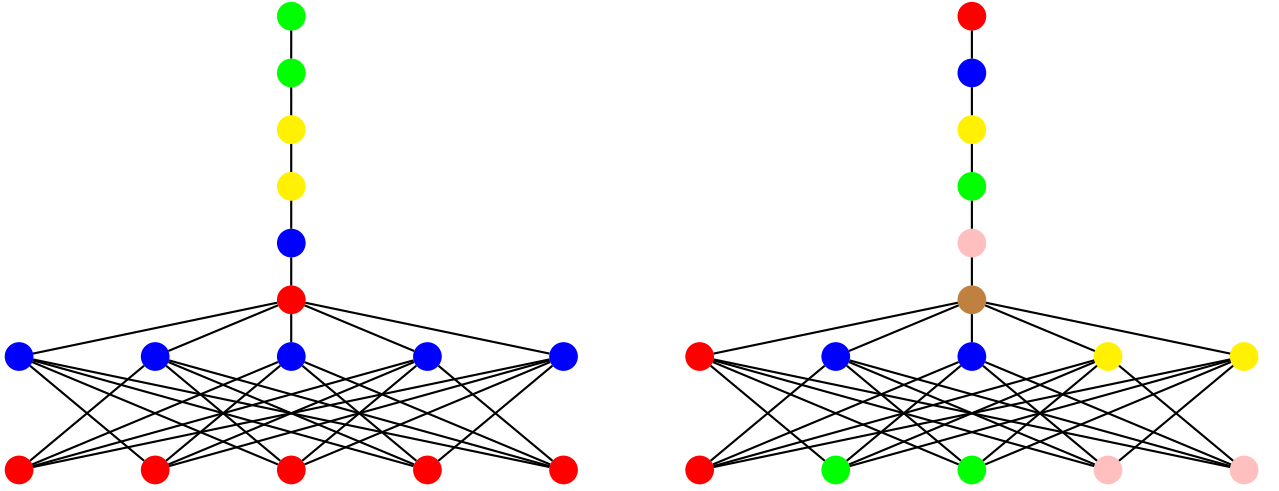	
		
	\end{proof}

	\begin{theorem}
		For $1 \leq a \leq b$, there exists a graph with $\chi _{\mu }(G) = a$ and $\chi _{\gp }(G) = b$ if and only if $a = b = 1$ or $2 \leq a \leq b$.
	\end{theorem}
	\begin{proof}
		Lemma~\ref{lem:chi_xi inequality} states that $\chi _{\mu }(G) \leq \chi _{\gp }(G)$. If $\chi _{\mu }(G) = 1$, then, as in the preceding proof, $G$ is a clique and $\chi _{\gp }(G) = 1$ as well, so assume that $a \geq 2$. Consider the graph $J(r,s)$ formed from the strong product $P_2 \boxtimes P_{2r}$ and a path $P_{2s-1}$ by joining a leaf of the path $P_{2s-1}$ to the two vertices $(1,2r)$ and $(2,2r)$. An optimal $\gp $-colouring requires at least $r+s$ colours by Corollary~\ref{lem:diameter bound} and a $\gp $-colouring using $r+s$ colours can be produced by colouring the subset $\{ 1,2\} \times \{ 2i-1,2i\} $ with different colours for $1 \leq i \leq r$ and then colouring the path $P_{2s-1}$ with $s$ colours as per Corollary~\ref{cor:paths}.
		
		We now show that $\chi _{\mu }(J(r,s)) = s+1$. By Corollary~\ref{lem:diameter bound} it takes at least $s$ colours to colour the path $P_{2s-1}$ in a $\mu $-colouring. Suppose for a contradiction that there is a $\mu $-colouring of $J(r,s)$ with $s$ colours. Then in order to colour the path $P_{2s-1}$, $s-1$ of the colour classes must contain two vertices of $P_{2s-1}$, and, as any pair of vertices from $P_{2s-1}$ is a maximal mutual-visibility set of $J(r,s)$, the remaining colour class must contain one vertex of $P_{2s-1}$ and all vertices of $\{ 1,2\} \times [2r]$, which is not a mutual-visibility set. Thus $\chi _{\mu }(J(r,s)) \geq s+1$. However, taking $\{ 1\} \times [2r]$ and the vertex $1$ of $P_{2s-1}$ as the first colour class, $\{ 2\} \times [2r] $ and the vertex $2$ (if $s \geq 2$) of the path $P_{2s-1}$ as the second colour class and then colouring the remaining $2s-3$ vertices of $P_{2s-1}$ greedily yields a $\mu $-colouring of $J(r,s)$ with $s+1$ colours.
		
		Therefore $J(b-a+1,a-1)$ satisfies \[ \chi _{\mu }J(b-a+1,a-1) = a \text{ and } \chi _{\gp }J(b-a+1,a-1) = b.\]    
	\end{proof}
	
	\section{Comparison with other graph parameters}\label{sec:graph parameters}
	
	In this section we explore the relationship between position colouring numbers and other well-known graph parameters, including the chromatic number, the clique cover number, the cochromatic number and the total domination number. We will also prove realisation results to compare these parameters. 
	
	As any clique is a monophonic position set, it follows that all three of $\chi _{\mu }(G)$, $\chi _{\gp }(G)$ and $\chi _{\mono }(G)$ are bounded above by the clique cover number $\theta (G)$.
	
	\begin{lemma}\label{lem:clique cover bound}
		For any graph $G$, if $\pi (G)$ is any one of $\mu (G)$, $\gp (G)$ or $\mono (G)$, then $\chi_{\pi }(G)\leq \theta (G)$.
	\end{lemma}

	\begin{theorem}
		There exists a connected graph with $\gp $-chromatic number $a$ and clique cover number $b$ if and only if $a = b = 1$ or $2 \leq a \leq b$.
	\end{theorem}
	\begin{proof}
		Suppose that there exists a graph $G$ with $\gp $-chromatic number $a$ and clique cover number $b$. By Lemma~\ref{lem:clique cover bound} we must have $a \leq b$. If $a = 1$, then by Lemma~\ref{lem:cliques} $G$ is a clique and we also have $b = 1$, so we can assume that $a \geq 2$. We use the same construction as in the proof of Theorem~\ref{thm:existence}; make a tree $T$ by adding $b-a$ leaves to the vertex $2a-2$ of the path $P_{2a-1}$. By Theorem~\ref{thm:block graph} we have $\chi _{\gp }(T) = a$. There are $b-a+1$ leaves adjacent to the vertex $2a-2$, $b-a$ of which cannot belong to a clique of order $\geq 2$ in a clique covering of $T$; it follows that the clique cover number of $T$ is $b-a+\left \lceil \frac{2a-1}{2} \right \rceil = b$. Thus the tree $T$ has the required properties.   
	\end{proof}
	
	Since any $\igp $- or $\imp $-colouring is proper, both of the associated position chromatic numbers are bounded below by the chromatic number.
	
	\begin{lemma}\label{lem:chi inequality}
		For any graph $G$, if $\pi (G)$ is any one of $\gp (G)$, $\mono (G)$ or $\mu (G)$, then $\chi _{\pi _i}(G) \geq \chi (G)$. 
	\end{lemma}

	\begin{theorem}\label{thm:diam three}
		If $G$ is a graph with $\diam(G)\leq 3$, and $\pi (G)$ is either $\mu (G)$ or $\gp(G)$, then $\chi_{\pi}(G) \leq \chi (G)$ and $\chi _{\pi _i}(G) = \chi({G})$.
	\end{theorem}
	\begin{proof} 
		For any graph $G$ with $\diam(G)\leq 3$, every independent set of $G$ is in general position. Therefore any proper colouring is a $\gp $-colouring and $\chi _{\gp }(G),\chi_{\mu }(G) \leq \chi (G)$. For independent general position and mutual visibility colourings equality follows from Lemma~\ref{lem:chi inequality}.
	\end{proof}

	\begin{theorem}
		For $a \geq 1$, there exists a graph $G$ with order $n$ and $\chi_{\igp}(G) = a$ if and only if $n\geq a$. Also, for any $1 \leq a \leq b$ there exists a graph $G$ with $\chi_{\gp}(G) = a$ and $\chi_{\igp}(G) = b$.
		\begin{proof}
			Clearly $\chi_{\igp}(G) = a$ implies $n \geq a$. If $n = a$, then $K_a$ will suffice. If $n > a$, form a graph $G(a)$ by attaching $n-a$ leaves to a vertex of a clique $K_a$. $G(a)$ has diameter two, so by Theorem~\ref{thm:diam three} we have $\chi _{\gp _i}(G(a)) = \chi (G(a)) = a$.
			
			For the second result, if $a = 1$, then the clique $K_b$ has $\chi _{\gp }(K_b) = 1$ and $\chi _{\gp _i}(K_b) = b$, so assume that $2 \leq a \leq b$. Consider the block graph of order $2b$ formed from two cliques of order $b-a+2$ and a path $P_{2a-2}$ by identifying one endpoint of the path with a vertex of the first clique and the other endpoint of $P_{2a-2}$ with a vertex of the other clique. By Theorem~\ref{thm:block graph} this graph has $\gp $-chromatic number $a$. Any independent general position set in this graph contains at most two vertices, and so the $\gp _i$-chromatic number is exactly $b$. 
		\end{proof}
	\end{theorem}

	\begin{theorem}
		There exists a graph with $\chi (G) = a$ and $\chi _{\igp }(G) = b$ if and only if $a = b = 1$ or $2 \leq a \leq b$. Also, there exist graphs with $\chi (G) = a$ and $\chi_{\imp }(G) = b$ for exactly the same range of $a$ and $b$.
	\end{theorem}
	\begin{proof}
		By Lemma~\ref{lem:chi inequality} we must have $a \leq b$. The only graphs with $\chi (G) = 1$ are the empty graphs, for which the $\igp $-chromatic number is also one, and if $a = b$, then $K_a$ suffices, so we can assume that $2 \leq a < b$. Let $G(a,b)$ be the graph formed by identifying a leaf of a path $P_{2b-a}$ with a vertex of a clique $K_a$. The chromatic number of this graph is given by $\chi (G(a,b)) = \omega (G(a,b)) = a$. Any optimal $\igp $-colouring of $G(a,b)$ must assign distinct colours to each vertex of $K_a$; each of these colours can recur once on the attached path, leaving $2b-2a-1$ vertices on the path still to be coloured. As each colour class can contain at most two vertices from the path, the remaining vertices of the path require a further $b-a$ colours, and it is possible to choose each of these colour classes to be an independent set, so $\chi _{\igp }(G(a,b)) = b$. This completes the proof of existence. The same argument works for $\imp $-colourings.
	\end{proof}

	As any clique is in general position, it also follows from Theorem~\ref{thm:diam three} that for graphs with diameter at most three, the $\gp $- and $\mono $-chromatic numbers are bounded above by the cochromatic number.
	
	\begin{corollary}
		If $\diam (G) \leq 3$, then $\chi _{\gp }(G) \leq \chi _{\mono }(G) \leq \z (G)$.
	\end{corollary}
	The Petersen graph in Figure~\ref{fig:Petersen} gives an example of a diameter two graph in which the $\gp $-chromatic number is strictly less than the cochromatic number. More generally, the graph $G = K_1 \vee (kK_k)$ of order $k^2+1$ and diameter two has $\chi _{\gp }(G) = 2$, but $\z (G) = k$, so the difference between the cochromatic number and the $\gp $-chromatic number can be arbitrarily large.
	
	We now bound the $\gp $-chromatic number by the total domination number for $K_4^-$-free graphs.
	
	\begin{lemma}\label{tot domination bound}
		In a diamond-free graph $G$ without isolated vertices we have \[ \chi_{\gp}(G) \leq \gamma _t(G).\]
	\end{lemma}
	\begin{proof}
		If $G$ is diamond-free, then no neighbourhood $N(u)$ contains an induced path of length two. Hence each open neighbourhood is in general position. The vertices of $G$ can be covered with at most $\gamma _t(G)$ such neighbourhoods. 
	\end{proof}
	
	The bound in Lemma~\ref{tot domination bound} is sharp for paths with order divisible by four; it would be interesting to characterise the graphs that achieve equality. However, we now show that $\gamma _t(G)-\chi _{\gp }(G)$ can be arbitrarily large for diamond-free graphs.  
	
	\begin{theorem}
		There exists a diamond-free graph $G$ with $\chi _{\gp }(G) = a$ and $\gamma _t(G) = b$ if and only if $2 \leq a \leq b$ or $a = 1$ and $b$ is even. If $G$ is also required to be connected, then for $a = 1$ we must have $b = 2$.
	\end{theorem}
	\begin{proof}
		The only graphs with $\chi _{\gp }(G) = 1$ are disjoint unions of cliques. The total domination number of any complete graph $K_n$, $n \geq 2$, is two, and so the total domination number of a disjoint union of cliques, none of which are $K_1$, must be even, and $\frac{b}{2}K_2$ has the required parameters. If we require $G$ to be connected and diamond-free, then only $b = 2$ is possible. 
		
		Now assume that $a \geq 2$. Recall that $\gamma _t(P_n)$ is $\frac{n}{2}$ if $n \equiv 0 \pmod{4}$, $\frac{n}{2}+1$ if $n \equiv 2 \pmod{4}$ and $\frac{n+1}{2}$ if $n$ is odd (see~\cite{HenYeo}).  Using Theorem~\ref{thm:block graph}, it follows that if $a = b = 2s$, then $\chi _{\gp }(P_{4s}) = \gamma _t(P_{4s}) = a$, and if $a = b= 2s+1$, then $\chi _{\gp }(P_{4s+1}) = \gamma _t(P_{4s+1}) = a$, so we can assume that $b > a$. If $a = 2$, then the graph formed by adding a leaf to each vertex of $K_b$ (that is, the corona product of $K_b$ and $K_1$) has the correct parameters, so we can also take $a \geq 3$. 
		
		We now show existence for the remaining values using spider graphs. Let $S(r,s)$ be the tree with one vertex $x$ of degree at least three such that $x$ is adjacent to an endvertex of $r \geq 2$ copies of $K_2$ and one path $P_s$. By Theorem~\ref{thm:block graph}, $\chi _{\gp }(S(r,s)) = \left \lceil \frac{s+3}{2} \right \rceil $ if $s \geq 2$ and $\chi _{\gp }(S(r,s)) = 3$ if $s \in \{ 0,1\} $. Hence if $a \geq 4$, we have $\chi _{\gp }(S(r,s)) = a$ if $s = 2a-3$ or $2a-4$. 
		
		Suppose that $K$ is a minimum total dominating set of $S(r,s)$. To dominate the leaves of $S(r,s)$ that belong to a copy of $K_2$ in $S(r,s)-x$, each of the $r$ support vertices of $S(r,s)$ that are neighbours of these leaves must belong to $K$ (and also dominate $x$). To dominate each of the support vertices, either the attached leaf or $x$ must be included in $K$; as $K$ is a minimum total dominating set, we can assume that $x \in K$, as this dominates all of these support vertices (as well as the first vertex of $P_s$ if $s \geq 1$). If $s \in \{ 0,1\} $ all the vertices are dominated, and $\gamma _t(S(r,s)) = r+1$, and so the graph $S(b-1,0)$ has $\chi _{\gp }(S(b-1,0)) = 3$ and $\gamma _t(S(b-1,0)) = b$. Hence we can now assume that $a \geq 4$. 
		
		If the neighbour of $x$ in $P_s$ belongs to $K$, then we can replace it by the third vertex of $P_s$, so it takes a further $\gamma _t(P_{s-1})$ vertices to dominate the remaining vertices of the path. Thus $\gamma _t(S(r,s)) = r+1+\gamma _t(P_{s-1})$. Thus if $a$ is even we have \[ \gamma _t(S(b-a+1,2a-3) = b-a+2+\gamma _t(P_{2a-4}) = b\] and if $a$ is odd then \[ \gamma _t(S(b-a+1,2a-4)) = b-a+2+\gamma _t(P_{2a-5}) = b.\]
	\end{proof}
	
	We note briefly that the diamond-free condition can be removed for mutual-visibility colourings, since any open neighbourhood is a mutual-visibility set.
	
	\begin{proposition}
		For any graph without isolated vertices, $\chi _{\mu }(G) \leq \gamma _t(G)$.
	\end{proposition}
	
	We finish this section with some bounds on the position colouring numbers in terms of size. We start with $\chi _{\gp _i}(G)$ and $\chi _{\mono _i}(G)$. Recall that the $a$-partite Tur\'{a}n graph with order $n$ is the complete $a$-partite graph on $n$ vertices with the parts as close together in order as possible; we will denote the size of this Tur\'{a}n graph by $t_a(n)$.
	
	\begin{theorem}
		If $\pi (G)$ is either $\gp (G)$ or $\mono (G)$, then the unique graph with $\chi _{\pi _i}(G) = a \geq 2$, order $n \geq a$ and largest possible size is the $a$-partite Tur\'{a}n graph with order $n$.
	\end{theorem}
	\begin{proof}
		If $\chi _{\pi _i}(G) = a$, then as each colour class is an independent set, $G$ must be an $a$-partite graph. Hence Tur\'{a}n's Theorem tells us that the size of $G$ is bounded above by $t_a(n)$. It is shown in~\cite{extremal position} that a subset $S$ of the vertices of the Tur\'{a}n graph is in general or monophonic position if and only if the vertices of $S$ all come from different partite sets (i.e. $S$ induces a clique) or all vertices of $S$ lie in the same partite set (i.e. $S$ is an independent set). Hence in any $\pi _i$-colouring of the Tur\'{a}n graph choosing each partite set as a colour class is optimal, so that the $\pi _i$-chromatic number of the Tur\'{a}n graph is $a$ and the upper bound $t_a(n)$ on the size is achieved.
	\end{proof}
	
	By contrast, a graph with given $\gp $- or $\mono $-chromatic numbers can have almost all edges present.  
	
	\begin{theorem}
		For $a \geq 2$, the largest size of a graph with order $n$ and $\chi _{\gp }(G) = a$ is at least ${n \choose 2}-\frac{(a-1)a(a+1)}{6}$ for $n \geq \frac{a(a+1)}{2}$ and for $n \geq 2a$ the largest size of a graph with $\chi _{\mono }(G) = a$ is at least ${n \choose 2}-(a-2)(2a-1)$ if $a \geq 3$ and ${n \choose 2}-1$ if $a = 2$.  
	\end{theorem}
	\begin{proof}
		Trivially $K_n^-$ is extremal for $\gp $- or $\mono $-chromatic number $a = 2$. For $n \geq \frac{a(a+1)}{2}$, consider the graph $H(n,a)$ formed by deleting the edges of a disjoint union $K_2 \dot \cup K_3 \dot \cup \cdots \dot \cup K_a$ from $K_n$. The graph $H(n,a)$ has the size given in the statement of the result and we claim that it has $\gp $-chromatic number $a$. We prove this by induction. The base case for $a = 2$ holds as $H(n,2) = K_n^-$. For $a \geq 3$, suppose that the result is true for $H(n,a-1)$ consider $H(n,a)$. We can decompose $H(n,a)$ into the join of $H(n,a-1)$ with the empty graph $aK_1$. The vertex set of $aK_1$ is in general position, but no general position set contains two or more vertices from $aK_1$ and a vertex outside of $aK_1$. Colouring $aK_1$ red and the vertices of $H(n,a-1)$ gives a $\gp $-colouring of $H(n,a)$ with $a$ colours. If there is any colouring of $H(n,a)$ with fewer than $a$ colours, then there must be two vertices of $aK_1$ with the same colour, say red. Then no vertices outside of $aK_1$ can be coloured red, so we can assume that all vertices of $aK_1$ are red, and a further $a-1$ colours are required to colour $H(n,a-1)$, so that $a$ colours is best possible. The result follows by induction.
		
		For the $\mono $-chromatic number and $a \geq 3$, consider the join of a cycle $C_{2a-1}$ with a clique $K_{n-2a+1}$. Any colour class of an $\mono $-colouring can contain at most two vertices of the cycle $C_{2a-1}$ and so $\chi _{\mono }(C_{2a-1} \vee K_{n-2a+1}) \geq a$. Conversely, choosing the first $a-1$ colour classes to be two vertices of the cycle and the final class to be the clique induced by the $K_{n-2a+1}$ and the remaining vertex of the cycle yields an $\mono $-colouring.
	\end{proof}
	We conjecture that for $a \geq 3$ the construction for $\gp $-chromatic number $a$ is extremal, and that the construction for $\mono $-chromatic number $a$ is extremal for sufficiently large $n$.

	%%%%%%%%%%%%%%%%%%%%%%
	%%%%%%%%%%%%%%%%%%%%%%
	\section{Position colouring of some graph classes}~\label{sec:classes}
	%%%%%%%%%%%%%%%%%%%%%%
	%%%%%%%%%%%%%%%%%%%%%%
	
	In this section we determine the $\gp $-chromatic numbers of some graph classes that have been investigated in the general position problem, namely complete multipartite graphs, line graphs of complete graphs and Kneser graphs. See~\cite{ghorbani-2019} for more details of the general position sets of such graphs.

	\begin{theorem}\label{thm:multipartite}
		The $\gp $-chromatic number of the complete multipartite graph $K_{n_1,n_2,\dots,n_r}$, where $r \geq 1$ and $n_1 \geq n_2 \geq \dots \geq n_r$, is given by \[\chi _{\gp}(K_{n_1,n_2,\dots,n_r}) = \min \{ r,r-i+n_{r-i+1}: 1 \leq i \leq r\} .\]
	\end{theorem}
	\begin{proof}
		Label the partite sets $V_1,V_2,\dots ,V_r$, where $|V_i| = n_i$ for $1 \leq i \leq r$. Suppose that $S$ is a general position set of $K_{n_1,n_2,\dots,n_r}$ that contains two vertices $x,y$ of a partite set $V_i$; then $S$ cannot contain a vertex $z$ from $V_j$ for $j \not = i$, for $x,z,y$ is a shortest path. It follows that every general position set of $K_{n_1,n_2,\dots,n_r}$ is either a clique or an independent set; conversely, any clique or independent set of $K_{n_1,n_2,\dots,n_r}$ is in general position, so that the $\gp $-chromatic number equals the cochromatic number for such graphs. The result now follows from Theorem 2.3 of~\cite{Gim}. 
	\end{proof}
	
	It was shown in~\cite{ghorbani-2019} that the general position sets of $L(K_n)$ correspond to an induced disjoint union of stars and triangles in $K_n$. Thus $\gp (L(K_n)) = n$ if $3|n$ and $\gp (L(K_n)) = n-1$ otherwise.
	\begin{theorem}
		For $n \geq 3$, the $\gp $-chromatic number of $L(K_n)$ is
		\[ \chi _{\gp }(L(K_n)) =\begin{cases}
			\frac{n}{2}+1, & \text{ if } n \in \{ 6,12\}  \\
			\frac{n+1}{2}, & \text{ if } n \equiv 1,5 \pmod 6,  \\
			\frac{n}{2}, & \text{ if } n \equiv 2,4 \pmod 6, \text{ or } n \equiv 0 \pmod 6 \text{ and } n \geq 18  \\
			\frac{n-1}{2}, & \text{ if } n \equiv 3 \pmod 6.  \\
		\end{cases}\]
	\end{theorem}
	\begin{proof}
		There are ${n \choose 2}$ vertices in $L(K_n)$, so by Lemma~\ref{lem:ni/xi bound} we have $\chi _{\gp }(L(K_n)) \geq \frac{n-1}{2}$ if $3|n$ and $\chi _{\gp }(L(K_n)) \geq \frac{n}{2}$ otherwise. Parity considerations allow us to improve this to $\chi _{\gp }(L(K_n)) \geq \frac{n}{2}$  for $n \equiv 0 \pmod 6$ and $\chi _{\gp }(L(K_n)) \geq \frac{n+1}{2}$ if $n \equiv 1,5 \pmod 6$. 
		
		In~\cite{Kirkmansystems} it is shown that for any $n \equiv 3 \pmod 6$ there exists a Kirkman triple system, also known as a resolvable Steiner triple system, i.e. a partition of the 2-subsets of $[n]$ into $\frac{n-1}{2}$ collections of $\frac{n}{3}$ disjoint triangles. It follows immediately that $\chi _{\gp }(L(K_n)) = \frac{n-1}{2}$ for $n \equiv 3 \pmod 6$. 
		
		If $n \equiv 4 \pmod 6$, let $x$ be a vertex of $K_n$ and colour the edges of $L(K_n-x)$ with a Kirkman triple system using $\frac{n-2}{2}$ colours, and colour the edges incident to $x$ with one new colour, making $\frac{n}{2}$ colours in total. If $n \equiv 5 \pmod 6$, let $x,y$ be vertices of $K_n$ and colour $K_n-\{ x,y\}$ with $\frac{n-3}{2}$ colours. Then colour the edges between $y$ and $K_n-\{ x,y\} $ with one new colour and the edges from $x$ to $K_n-x$ with another new colour, giving $\frac{n+1}{2}$ colours. 
		
		If $n \equiv 2 \pmod 6$, take a Kirkman triple system on $K_{n+1}$ and then delete a single vertex $x$ from each of the $\frac{n}{2}$ collections of triangles. This leaves a decomposition of $E(K_n)$ into $\frac{n}{2}$ collections of $\frac{n-2}{3}$ disjoint triangles and a $P_2$. Similarly, for $n \equiv 1 \pmod 6$, add two new vertices $x,y$ to $K_n$ to get a clique $K_{n+2}$. Take a Kirkman triple system on $K_{n+2}$ and delete $x,y$ from each of the $\frac{n+1}{2}$ collections of triangles; again, this leaves a collection of general position sets, each of which is either a disjoint union of triangles or a disjoint union of triangles and a copy of $2K_2$.
		
		This leaves the case $n \equiv 0 \pmod 6$. Delete three vertices $x,y,z$ to get a complete graph $K_{n-3}$ and take a Kirkman triple system on $K_{n-3}$, which uses $\frac{n-4}{2}$ colours. Now in $K_n$ colour the triangle on $x,y,z$ with one of the previously used colours. Then use a new colour for the edges from $x$ to $K_n-\{x,y,z\} $, a second new colour for the edges from $y$ to $K_n-\{ y,z\} $ and a third new colour for the edges from $z$ to $K_n-\{ z\}$. This uses $\frac{n}{2}+1$ colours, which is one greater than the lower bound. It can be checked~\cite{erskine} that $\frac{n}{2}$ colours is impossible for $n \in \{ 6,12\} $. However, it is possible to meet the lower bound $\frac{n}{2}$ when $n \geq 18$ and $6|n$. For any such value there exists a nearly Kirkman triple system (see pages 351 and 352 of~\cite{triplesystems}), which is equivalent to a decomposition of $E(K_n)$ into $\frac{n}{2}-1$ collections of disjoint triangles and one copy of $3K_2$, which is also a general position set.  
	\end{proof}

	The vertex set of the Kneser graph $K(n,2)$ consists of all 2-subsets of $[n]$, and two subsets $\{ a,b\} $ and $\{ a',b'\} $ are adjacent if and only if they are disjoint. The argument in~\cite{ghorbani-2019} shows that in the Kneser graph $K(n,2)$, there are four types of maximal general position set: i) an independent set of three vertices, induced by the 2-subsets of any 3-subset of $[n]$, ii) a copy of $3K_2$ induced by the 2-subsets of any 4-subset, iii) a clique (i.e.\ a collection of disjoint 2-subsets) of order $\left \lfloor \frac{n}{2}\right \rfloor $ and iv) an independent set of $n-1$ vertices of the form given by the Erd\H{o}s-Ko-Rado Theorem, in which all the 2-subsets have a common element. Hence $\gp (K(n,2)) = 6$ for $n \in \{ 5,6\}$ and $\gp (K(n,2)) = n-1$ for $n \geq 7$.
	
	\begin{theorem}
		For $n \geq 5$, $\chi _{\gp }(K(n,2)) = n-3$.
	\end{theorem}
	\begin{proof}
		Figure~\ref{fig:Petersen} shows that $\chi _{\gp }(K(5,2)) = 2$. It follows from Lemma~\ref{lem:ni/xi bound} that \newline $\chi _{\gp }(K(n,2)) \geq \left \lceil \frac{n}{2} \right \rceil $ for $n \geq 6$. That $\chi _{\gp }(K(n,2)) \leq n-3$ follows by colouring the six 2-subsets of $[4]$ red, then by dividing the remaining vertices into $n-4$ independent sets, e.g. all those 2-subsets containing a 5, the 2-subsets containing a 6, etc. 
		
		Our bounds show that $\chi _{\gp }(K(n,2)) = n-3$ holds for $n = 6,7$, and that $4 \leq \chi _{\gp }(K(8,2)) \leq 5$ and $5 \leq \chi _{\gp }(K(9,2)) \leq 6$. Equality $\chi _{\gp }(K(8,2)) = 4$ can hold only if we can partition the vertex set into independent sets of seven vertices. However, this is impossible, since it is known that $\chi (K(n,2)) = n-2$ (see~\cite{Lovasz}), so $\chi _{\gp }(K(8,2)) = 5$.
		
		We claim that for $n \geq 9$ we can choose an optimal $\gp $-colouring of $K(n,2)$ that contains a colour class that is a subset of an independent set of type iv). Otherwise, all colour classes contain at most $\max \{ 6,\left \lfloor \frac{n}{2} \right \rfloor \} $ vertices and for $n \geq 10$ we have $(n-3)\max \{ 6,\left \lfloor \frac{n}{2} \right \rfloor \} < {n \choose 2}$. If $\chi _{\gp }(K(9,2)) = 5$ and if there is no colour class that is a subset of the form iv), then we can cover at most $6 \times 5$ vertices of $K(9,2)$ using five colours, so whether $\chi _{\gp }(K(9,2))$ is five or six there must be a colour class of the form claimed. 
		
		We now proceed by induction starting at $r = 8$. Assume that $\chi _{\gp }(K(r,2)) = r-3$ for $8 \leq r < n$. Take an optimal $\gp $-colouring of $K(n,2)$ with the red vertices an independent set all containing the symbol $n$. Now we can recolour any vertex of $K(n,2)$ that contains the symbol $n$ red and still have an optimal $\gp $-colouring. The colouring of the vertices not containing $n$ induces a $\gp $-colouring of $K(n-1,2)$, and so by induction there must be $\geq n-4$ further colours, establishing the theorem.
	\end{proof}

	\section{Small and large position colouring numbers}\label{sec:small&large}
	
	In this section we briefly discuss characterisations of graphs with very large or small values of the position colouring numbers. We dealt with the extreme cases of position colouring numbers equal to one or $n$ in Lemma~\ref{lem:cliques}. We start with characterising some graphs with large general position numbers relative to the order. Lemma~\ref{lem:ni/xi bound} implies that for any fixed $a$, there are only finitely many graphs with $\chi _{\pi }(G) = n-a$ if $\pi $ is one of $\gp $, $\mono $ or $\mu $. For independent $\gp $-colourings we use the following result that follows by dividing a shortest path of length $\diam (G)$ into independent sets.

	\begin{lemma}\label{gpi diam bound}
		For any graph $G$, if $\pi $ is one of $\gp $ or $\mono $, then
		\[ \chi _{\pi _i}(G) \leq n-\left \lfloor \frac{\diam ^*(G)+1}{2} \right \rfloor  .\] 
	\end{lemma}
	Identifying an endpoint of a path with a vertex of a complete graph shows that Lemma~\ref{gpi diam bound} is tight for all diameters.
	
	\begin{theorem}
		If $G$ has order $n$, then
		\begin{itemize}
			\item $\chi _{\pi }(G) = n-1$ if and only if $G$ is one of $P_2$, $P_3$ or $2K_1$ when $\pi $ is either $\gp $ or $\mono $, 
			\item the only graphs with $\chi _{\gp }(G) = n -2$ are i) graphs with order three that are disjoint unions of cliques, ii) graphs with order four that are not disjoint unions of cliques and iii) $P_5$,
			\item the graphs with $\chi _{\mono }(G) = n-2$ are those with $\chi _{\gp}(G) =n-2$ as well as the cycle $C_5$, and
			\item $\chi _{\gp _i}(G) = n-1$ if and only if $G$ has clique number $n-1$.
		\end{itemize}
		
	\end{theorem}
	\begin{proof}
		The first part follows trivially from Lemma~\ref{lem:ni/xi bound}. Let $\pi $ be either $\gp $ or $\mono $ and suppose that $\chi _{\pi }(G) = n-2$. No graph with order $n \geq 6$ can have $\pi $-chromatic number $n-2$. The result for order three follows by Lemma~\ref{lem:cliques} and, by the same lemma, a graph with order four has $\pi $-chromatic number two if and only if it is not a disjoint union of cliques. A graph with order five and $\pi $-chromatic number three must be connected. We have $\chi _{\gp }(P_5) = \chi _{\mono }(P_5) = \chi _{\mono }(C_5) = 3$, but any other connected graph with order five has a general or monophonic position set of order at least three, and the remaining vertices can be coloured with a single colour, so that these are the only solutions.   
		
		Now suppose that $G$ is a graph with $\gp _i$-chromatic number $n-1$. If $\diam ^*(G) \geq 3$, then by Lemma~\ref{gpi diam bound} at most $n-2$ colours would be needed in a $\gp _i$-colouring of $G$. Hence Theorem~\ref{thm:diam three} shows that if $\chi _{\gp _i}(G) = n-1$, then $\chi _{\gp _i}(G) = \chi (G) = n-1$, and hence $\omega (G) = n-1$. Conversely any graph with clique number $n-1$ obviously has $\gp _i$-chromatic number $n-1$.  
	\end{proof}
	
	At the other extreme, we now discuss graphs with position chromatic number two. If $\chi _{\gp _i}(G) = 2$, then $G$ must be bipartite, and a bipartite graph will have $\chi _{\gp _i}(G) = 2$ if and only if it has diameter two or three and $\chi _{\mono _i}(G) = 2$ if and only if it has monophonic diameter two or three. The vertex set of any graph with $\chi _{\gp }(G) = 2$ or $\chi _{\mono }(G) = 2$ can be divided into two independent unions of cliques. Let us write $A^i = \{ A_1^i,\dots ,A_{r_i}^i\} $, where $i \in \mathbb{Z}_2$, for the two independent unions of cliques. 
	
	\begin{theorem}\label{thm:chromatic no 2}
		A graph $G$ has $\chi _{\gp }(G) = 2$ if and only if $2 \leq \diam ^*(G) \leq 3$ and $V(G)$ can be divided into two independent unions of cliques such that, for $i \in \mathbb{Z}_2$, if a vertex $w$ in $A^i$ has neighbours $u,v$ in two distinct cliques $A_x^{i+1},A_y^{i+1}$, then $d(u,v') = d(u',v) = 2$ for each $u' \in A_x^{i+1}$ and $v' \in A_y^{i+1}$. The analogous result for $\chi _{\mono }(G) = 2$ holds if we replace `shortest path' by `induced path' and `diameter' by `monophonic diameter'. 
	\end{theorem}
	\begin{proof}
		We prove the result for $\chi _{\gp }(G) = 2$; the reasoning for $\chi _{\mono }(G) = 2$ is similar. The result will follow if we prove the statement for connected graphs. Let $G$ be a connected graph with $\chi _{\gp }(G) = 2$. That $\diam (G) \leq 3$ follows from Corollary~\ref{lem:diameter bound}. Let $A^0,A^1$ be the independent unions of cliques in $G$ that will represent the two colour classes as per our discussion. No shortest path in $G$ of length two can contain three vertices from the same $A^i$, so we need only make sure that any shortest path of length three does not contain three vertices from the same $A^i$. This will only happen if there is a shortest path of the form $x_1,x_2,w,y_1$, where $w \in A^{i+1}$, $x_1,x_2 \in A_x^i$, $y_1 \in A_y^i$. 
	\end{proof}
	
	It follows from Theorem~\ref{thm:chromatic no 2} that one family of graphs with $\gp $- and $\mono $-chromatic number two is blow-ups of bipartite graphs with diameter two or three. We mention two other such families. A graph $G$ is a \emph{split graph} if $V(G)$ can be partitioned into a clique and an independent set. If $G$ is a graph and $\overline{G}$ its complement, then the {\em complementary prism} $G\overline{G}$ of $G$ is the graph formed from the disjoint union of $G$ and $\overline{G}$ by adding a perfect matching between the corresponding vertices of $G$ and $\overline{G}$. For example, the Petersen graph is the complementary prism of $C_5$. Theorem~\ref{thm:chromatic no 2} gives the following corollary.
	
	\begin{corollary}
		If $G$ is a connected non-complete split graph, then $\chi _{\gp }(G) = \chi _{\mono}(G) = 2$ and also for the complementary prism $G\overline{G}$ of $G$, $\chi _{\gp }(G\overline{G})=2$.   
	\end{corollary}
	
	%We mention also that if $G$ is a split graph such that both $G$ and $\overline{G}$ are connected, then we have $\chi _{\mono}(G\overline{G}) = 4$, but we omit the proof.

	\section{$\gp $-colourings of Cartesian products}\label{sec:cartesian}
	
	In this section we investigate the general position chromatic numbers of Cartesian products. These graphs have been investigated extensively in the general position problem, see for example~\cite{Cartesian} and~\cite{torus}. Here we will focus on products of paths and cycles. In some cases we derive exact expressions, in others we will be content with finding the asymptotic order.
	
	\begin{proposition}\label{cartesian product bound} 
		If $G$ and $H$ are graphs with orders $n_1$ and $n_2$ respectively, then the $\gp $-chromatic number of the Cartesian product of $G$ and $H$ is bounded by
		\[ \frac{n_1n_2}{\gp (G \cp H) } \leq \chi _{\gp }(G \cp H) \leq \min \{ n_1\chi _{\gp }(H),n_2\chi _{\gp }(G)\}.\] If $\gp (G \cp H)|n_1n_2$ and there is a vertex of $G \cp H$ not contained in a maximum general position set of $G \cp H$, then the lower bound can be improved to $\frac{n_1n_2}{\gp (G\cp H)}+1$. 
	\end{proposition}
	\begin{proof}
		The lower bound follows immediately from Lemma~\ref{lem:ni/xi bound}. In the case that $\gp (G \cp H)|n_1n_2$ equality is possible in the lower bound only if every vertex of $G \cp H$ is contained in a maximum general position set. For the upper bound, observe that if $S$ is a general position set of $G$, then for any $h \in V(H)$, $S \times \{ h\} $ is in general position in the layer $G^h$ of $G \cp H$ ; therefore taking $n_2$ copies of an optimal $\gp $-colouring of $G$, one for each $G$-factor, yields a $\gp $-colouring of $G \cp H$.
	\end{proof}
	
	We show that the lower bound in Proposition~\ref{cartesian product bound} is asymptotically sharp for Cartesian products of paths and exhibit products for which it is exact. Recall that we label the vertices of the path $P_n$ by the elements of $[n]$, so that an element of $P_m \cp P_n$ will be written as $(i,j)$ for some $1 \leq i \leq m$ and $1 \leq j \leq n$.
	
	\begin{theorem}
		The $\gp $-chromatic number of $P_2 \cp P_n$ for $n \geq 3$ is
		\[ \chi _{\gp }(P_2 \cp P_n) =\begin{cases}
			2r, & \text{ if } n = 3r,  \\
			2r+1, & \text{ if } n = 3r+1, \\
			2r+2, & \text{ if } n = 3r+2. \\
		\end{cases}\]
	\end{theorem}
	\begin{proof}
		The $\gp $-number of $P_2 \cp P_n$ is three~\cite{Cartesian}. By the lower bound of Proposition~\ref{cartesian product bound} we have $\chi _{\gp }(P_2 \cp P_n) \geq \frac{2n}{3}$, which shows that the expression in the statement of the theorem is a lower bound for $\chi _{\gp }(P_2 \cp P_n)$. It follows from Lemma~\ref{tot domination bound} that $\chi _{\gp }(P_2 \cp P_n) \leq \gamma _t(P_2 \cp P_n)$. It is shown in~\cite{Grav} that the total domination number of $P_2 \cp P_n$ is given by $\gamma _t(P_2 \cp P_n) = 2 \left \lfloor \frac{n+2}{3} \right \rfloor $. For $n$ with remainder $0$ or $2$ on division by $3$ this matches our lower bound. If $n = 3r+1$, let $n$ be a leaf of $P_n$; then the subgraph $P_2 \cp (P_n-\{ n\} )$ can be coloured by $2r$ general position sets by the result for orders divisible by three, with $\{ (1,n),(2,n)\} $ giving the final colour class.
	\end{proof}
	
	\begin{theorem}
		For $n \geq 3$, the $\gp $-chromatic number of $P_3 \cp P_n$ is \[ \chi _{\gp }(P_3 \cp P_n) = \frac{5n}{6}+O(1).\] If $12|n$, then $\chi _{\gp }(P_3 \cp P_n) = \frac{5n}{6}$.
	\end{theorem}
	\begin{proof}
		Each general position set of $P_3\cp P_n$ of order four contains two vertices of the central $P_n$ layer. Therefore there are at most $\left \lfloor \frac{n}{2} \right \rfloor $ colour classes in a $\gp $-colouring of $P_3 \cp P_n$ containing four vertices. Therefore \[ \chi _{\gp }(P_3 \cp P_n) \geq \left \lfloor \frac{n}{2} \right \rfloor +\frac{3n-4\left \lfloor \frac{n}{2} \right \rfloor}{3} = \frac{5n}{6}+O(1).\] 
		Suppose that $n = 4q$. For $0 \leq i \leq q-1$, colour the neighbourhood of the vertex $(2,4i+2)$ with colour $L_i$ and the neighbourhood of vertex $(2,4i+3)$ with colour $R_i$. This leaves the vertices $(r,4j+s)$ to be coloured, where $r \in \{ 1,3\} $, $0 \leq j \leq q-1$ and $s \in \{ 1,4\} $. At most four of these vertices cannot be included in a colour class of order three. If $12|n$, let $n = 12q'$ and for $0 \leq i \leq q'-1$ colour the remaining vertices as follows:
		\begin{itemize}
			\item vertices $(1,12i+1),(1,12i+5),(3,12i+4)$ with colour $A_i$,
			\item vertices $(1,12i+4),(3,12i+1),(3,12i+5)$ with colour $B_i$,
			\item vertices $(1,12i+8),(1,12i+12),(3,12i+9)$ with colour $C_i$,
			\item vertices $(1,12i+9),(3,12i+8),(3,12i+12)$ with colour $D_i$.
		\end{itemize}
		Each colour class is in general position. When $12|n$, we have used exactly $\frac{5n}{6}$ colours. Such a colouring is shown in Figure~\ref{fig:P_3 cp P_n}. If $n = 12q'+r$, $1 \leq r < 12$, then a subgraph $P_3 \cp P_{12q'}$ can be coloured with $10q'$ colours and a fixed amount of colours suffices for the remaining vertices.
	\end{proof}

	\begin{figure}\centering
		\begin{tikzpicture}[x=0.2mm,y=-0.2mm,inner sep=0.1mm,scale=1.10,
			thick,vertex/.style={circle,draw,minimum size=10,font=\tiny,fill=white},edge label/.style={fill=white}]
			\tiny
			
			\node at (-125,-75) [vertex] (11) {g};
			\node at (-75,-75) [vertex] (12) {a};
			\node at (-25,-75) [vertex] (13) {b};
			\node at (25,-75) [vertex] (14) {h};
			\node at (75,-75) [vertex] (15) {g};
			\node at (125,-75) [vertex] (16) {c};,
			\node at (175,-75) [vertex] (17) {d};
			\node at (225,-75) [vertex] (18) {i};
			\node at (275,-75) [vertex] (19) {j};
			\node at (325,-75) [vertex] (110) {e};
			\node at (375,-75) [vertex] (111) {f};
			\node at (425,-75) [vertex] (112) {i};

			\node at (-125,-25) [vertex] (21) {a};
			\node at (-75,-25) [vertex] (22) {b};
			\node at (-25,-25) [vertex] (23) {a};
			\node at (25,-25) [vertex] (24) {b};
			\node at (75,-25) [vertex] (25) {c};
			\node at (125,-25) [vertex] (26) {d};,
			\node at (175,-25) [vertex] (27) {c};
			\node at (225,-25) [vertex] (28) {d};
			\node at (275,-25) [vertex] (29) {e};
			\node at (325,-25) [vertex] (210) {f};
			\node at (375,-25) [vertex] (211) {e};
			\node at (425,-25) [vertex] (212) {f};
			
			\node at (-125,25) [vertex] (31) {h};
			\node at (-75,25) [vertex] (32) {a};
			\node at (-25,25) [vertex] (33) {b};
			\node at (25,25) [vertex] (34) {g};
			\node at (75,25) [vertex] (35) {h};
			\node at (125,25) [vertex] (36) {c};,
			\node at (175,25) [vertex] (37) {d};
			\node at (225,25) [vertex] (38) {j};
			\node at (275,25) [vertex] (39) {i};
			\node at (325,25) [vertex] (310) {e};
			\node at (375,25) [vertex] (311) {f};
			\node at (425,25) [vertex] (312) {j};

			\path
			
			(11) edge (12)
			(12) edge (13)
			(13) edge (14)
			(14) edge (15)
			(15) edge (16)
			(16) edge (17)
			(17) edge (18)
			(18) edge (19)
			(19) edge (110)
			(110) edge (111)
			(111) edge (112)
			
			(21) edge (22)
			(22) edge (23)
			(23) edge (24)
			(24) edge (25)
			(25) edge (26)
			(26) edge (27)
			(27) edge (28)
			(28) edge (29)
			(29) edge (210)
			(210) edge (211)
			(211) edge (212)
			
			(31) edge (32)
			(32) edge (33)
			(33) edge (34)
			(34) edge (35)
			(35) edge (36)
			(36) edge (37)
			(37) edge (38)
			(38) edge (39)
			(39) edge (310)
			(310) edge (311)
			(311) edge (312)
			
			(11) edge (21)
			(21) edge (31)
			
			(12) edge (22)
			(22) edge (32)
			
			(13) edge (23)
			(23) edge (33)
			
			(14) edge (24)
			(24) edge (34)
			
			(15) edge (25)
			(25) edge (35)
			
			(16) edge (26)
			(26) edge (36)
			
			(17) edge (27)
			(27) edge (37)
			
			(18) edge (28)
			(28) edge (38)
			
			(19) edge (29)
			(29) edge (39)
			
			(110) edge (210)
			(210) edge (310)
			
			(111) edge (211)
			(211) edge (311)
			
			(112) edge (212)
			(212) edge (312)
			
			;

		\end{tikzpicture}
		\caption{A $\gp $-colouring of a $3 \times 12$ Cartesian grid with colours represented by letters}
		\label{fig:P_3 cp P_n}
	\end{figure}
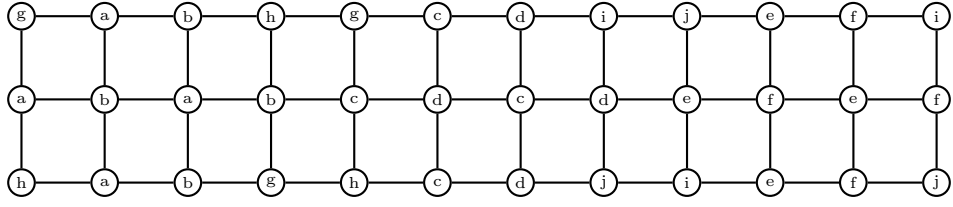

	Products with paths of length three already allow for an efficient packing of general position sets of order four.
	
	\begin{theorem}
		If $n \geq 4$, then $\chi _{\gp }(P_4 \cp P_n) = n+1$.
	\end{theorem}
	\begin{proof}
		For $1 \leq j \leq n-2$ the set containing the four vertices $(1,j+1)$, $(2,j)$, $(3,j+2)$ and $(4,j+1)$ is in general position. These sets are disjoint and leave eight vertices to be coloured, which can be done with the three new colour classes $\{ (2,n),(3,1)\} $, $\{ (4,1),(4,n),(3,2)\} $ and $\{ (1,1),(1,n),(2,n-1)\} $. This colours the $4 \times n$ grid with $n+1$ colours. A $\gp $-colouring with $n$ colours is impossible by Proposition~\ref{cartesian product bound}, as the corner vertices are not contained in any general position set of order four, so $n+1$ colours is optimal. An example is shown in Figure~\ref{fig:P_4 cp P_n}. 
	\end{proof}

	\begin{figure}\centering
		\begin{tikzpicture}[x=0.2mm,y=-0.2mm,inner sep=0.1mm,scale=1.30,
			thick,vertex/.style={circle,draw,minimum size=10,font=\tiny,fill=white},edge label/.style={fill=white}]
			\tiny
			
			\node at (-125,-75) [vertex,color=yellow] (11) {};
			\node at (-75,-75) [vertex,color=red] (12) {};
			\node at (-25,-75) [vertex,color=blue] (13) {};
			\node at (25,-75) [vertex,color=green] (14) {};
			\node at (75,-75) [vertex,color=pink] (15) {};
			\node at (125,-75) [vertex,color=yellow] (16) {};
			
			\node at (-125,-25) [vertex,color=red] (21) {};
			\node at (-75,-25) [vertex,color=blue] (22) {};
			\node at (-25,-25) [vertex,color=green] (23) {};
			\node at (25,-25) [vertex,color=pink] (24) {};
			\node at (75,-25) [vertex,color=yellow] (25) {};
			\node at (125,-25) [vertex,color=orange] (26) {};
			
			\node at (-125,25) [vertex,color=orange] (31) {};
			\node at (-75,25) [vertex,color=brown] (32) {};
			\node at (-25,25) [vertex,color=red] (33) {};
			\node at (25,25) [vertex,color=blue] (34) {};
			\node at (75,25) [vertex,color=green] (35) {};
			\node at (125,25) [vertex,color=pink] (36) {};
			
			\node at (-125,75) [vertex,color=brown] (41) {};
			\node at (-75,75) [vertex,color=red] (42) {};
			\node at (-25,75) [vertex,color=blue] (43) {};
			\node at (25,75) [vertex,color=green] (44) {};
			\node at (75,75) [vertex,color=pink] (45) {};
			\node at (125,75) [vertex,color=brown] (46) {};

			\path
			
			(11) edge (12)
			(12) edge (13)
			(13) edge (14)
			(14) edge (15)
			(15) edge (16)
			
			(21) edge (22)
			(22) edge (23)
			(23) edge (24)
			(24) edge (25)
			(25) edge (26)
			
			(31) edge (32)
			(32) edge (33)
			(33) edge (34)
			(34) edge (35)
			(35) edge (36)
			
			(41) edge (42)
			(42) edge (43)
			(43) edge (44)
			(44) edge (45)
			(45) edge (46)
			
			(11) edge (21)
			(21) edge (31)
			(31) edge (41)
			
			(12) edge (22)
			(22) edge (32)
			(32) edge (42)
			
			(13) edge (23)
			(23) edge (33)
			(33) edge (43)
			
			(14) edge (24)
			(24) edge (34)
			(34) edge (44)
			
			(15) edge (25)
			(25) edge (35)
			(35) edge (45)
			
			(16) edge (26)
			(26) edge (36)
			(36) edge (46)

			;

		\end{tikzpicture}
		\caption{A $\gp $-colouring of a $4 \times 6$ Cartesian grid}
		\label{fig:P_4 cp P_n}
	\end{figure}
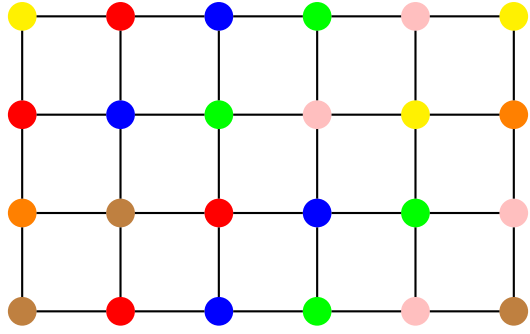

	Using the lower bound from Proposition~\ref{cartesian product bound} and the upper bound for the total domination number of the Cartesian product of paths from~\cite{Grav} we obtain the following result for Cartesian products of long paths.
	\begin{theorem}
		If $n_1,n_2 \geq 16$, then 
		\[ \frac{n_1n_2}{4} \leq \chi _{\gp }(P_{n_1} \cp P_{n_2} ) \leq \left \lfloor \frac{(n_1+2)(n_2+2)}{4} \right \rfloor -4. \]
	\end{theorem}
	It follows that $\chi _{\gp }(P_{n_1} \cp P_{n_2}) \approx \frac{n_1n_2}{4}$ for $n_1,n_2 \geq 16$ and the total domination bound in Theorem~\ref{tot domination bound} is asymptotically tight for grids. We give an example to show that the linear terms can be improved. 
	
	\begin{theorem}\label{thm:gridtesselation}
		If $n_1 \equiv 3 \pmod{4}$ and $n_2 \equiv 0 \pmod{4}$, then
		\[ \chi _{\gp } (P_{n_1} \cp P_{n_2} ) \leq \begin{cases}
			\frac{n_1n_2}{4}+\frac{n_1}{6} + O(1), & \text{ if } n_2 < 2n_1-2,  \\
			\frac{n_1n_2}{4}+ \frac{n_2}{12} + O(1), & \text{ if } n_2 > 2n_1-2. \\
		\end{cases}\]
		When $n_2 = 2n_1-2$, $\chi _{\gp }(P_{n_1} \cp P_{n_2}) = \frac{n_1n_2}{4}+1$.
	\end{theorem}
	\begin{proof}
		As the Cartesian grid is diamond-free, any vertex neighbourhood is a general position set. Consider the open neighbourhoods of vertices  $(i,j)$ such that either 
		\begin{itemize}
			\item $i \equiv 0 \pmod 4$, $j \equiv 0,1 \pmod 4$, $j \not \in \{ 1,n_2\} $, or
			\item $i \equiv 2 \pmod 4$, $j \equiv 2,3 \pmod 4$.
		\end{itemize}
		These neighbourhoods do not overlap and are all general position sets of order four. Colour each of these neighbourhoods with a new colour; an example is shown in Figure~\ref{fig:P_7 cp P_8}.
		
		The $2n_1+n_2-6$ vertices that remain uncoloured consist of the four corner vertices, $\frac{n_2}{2}-2$ vertices in each of $\{ 1\} \times V(P_{n_2})$ and $\{ n_1\} \times V(P_{n_2})$, and $n_1 - 3$ vertices in each of $V(P_n) \times \{ 1,2\}$ and $V(P_n) \times \{ n_2-1,n_2\} $ belonging to the closed neighbourhoods of vertices $(4r,1)$ and $(4r,n_2)$, $1 \leq r \leq \left \lfloor \frac{n_1}{4} \right \rfloor $. Another $\min \{ \frac{n_2}{2}-2,n_1-3\} $ of the remaining non-corner vertices can be coloured with general position sets of order four, and then the remaining ones with general position sets of order three until at most two uncoloured vertices remain. When $n_2 = 2n_1-2$, the number of colours used matches the adjusted lower bound in Proposition~\ref{cartesian product bound}.
	\end{proof}

	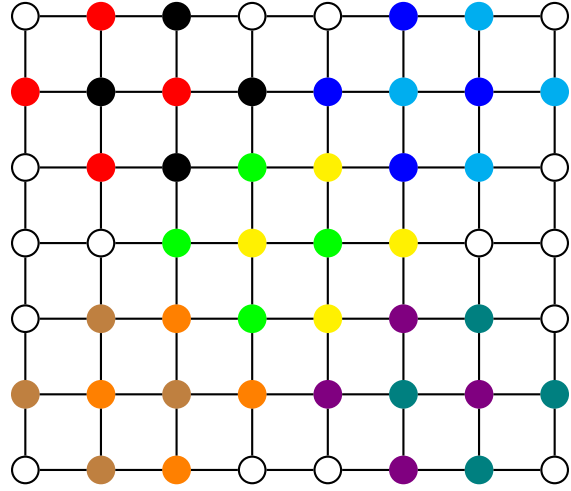
\begin{figure}\centering
		\begin{tikzpicture}[x=0.2mm,y=-0.2mm,inner sep=0.1mm,scale=1.0,
			thick,vertex/.style={circle,draw,minimum size=10,font=\tiny,fill=white},edge label/.style={fill=white}]
			\tiny
			
			\node at (-125,-75) [vertex] (11) {};
			\node at (-75,-75) [vertex,color=red] (12) {};
			\node at (-25,-75) [vertex,color=black] (13) {};
			\node at (25,-75) [vertex] (14) {};
			\node at (75,-75) [vertex] (15) {};
			\node at (125,-75) [vertex,color=blue] (16) {};
			\node at (175,-75) [vertex,color=cyan] (17) {};
			\node at (225,-75) [vertex] (18) {};
			
			\node at (-125,-25) [vertex,color=red] (21) {};
			\node at (-75,-25) [vertex,color=black] (22) {};
			\node at (-25,-25) [vertex,color=red] (23) {};
			\node at (25,-25) [vertex,color=black] (24) {};
			\node at (75,-25) [vertex,color=blue] (25) {};
			\node at (125,-25) [vertex,color=cyan] (26) {};
			\node at (175,-25) [vertex,color=blue] (27) {};
			\node at (225,-25) [vertex,color=cyan] (28) {};
			
			\node at (-125,25) [vertex] (31) {};
			\node at (-75,25) [vertex,color=red] (32) {};
			\node at (-25,25) [vertex,color=black] (33) {};
			\node at (25,25) [vertex,color=green] (34) {};
			\node at (75,25) [vertex,color=yellow] (35) {};
			\node at (125,25) [vertex,color=blue] (36) {};
			\node at (175,25) [vertex,color=cyan] (37) {};
			\node at (225,25) [vertex] (38) {};
			
			\node at (-125,75) [vertex] (41) {};
			\node at (-75,75) [vertex] (42) {};
			\node at (-25,75) [vertex,color=green] (43) {};
			\node at (25,75) [vertex,color=yellow] (44) {};
			\node at (75,75) [vertex,color=green] (45) {};
			\node at (125,75) [vertex,color=yellow] (46) {};
			\node at (175,75) [vertex] (47) {};
			\node at (225,75) [vertex] (48) {};
			
			\node at (-125,125) [vertex] (51) {};
			\node at (-75,125) [vertex,color=brown] (52) {};
			\node at (-25,125) [vertex,color=orange] (53) {};
			\node at (25,125) [vertex,color=green] (54) {};
			\node at (75,125) [vertex,color=yellow] (55) {};
			\node at (125,125) [vertex,color=violet] (56) {};
			\node at (175,125) [vertex,color=teal] (57) {};
			\node at (225,125) [vertex] (58) {};
			
			\node at (-125,175) [vertex,color=brown] (61) {};
			\node at (-75,175) [vertex,color=orange] (62) {};
			\node at (-25,175) [vertex,color=brown] (63) {};
			\node at (25,175) [vertex,color=orange] (64) {};
			\node at (75,175) [vertex,color=violet] (65) {};
			\node at (125,175) [vertex,color=teal] (66) {};
			\node at (175,175) [vertex,color=violet] (67) {};
			\node at (225,175) [vertex,color=teal] (68) {};
			
			\node at (-125,225) [vertex] (71) {};
			\node at (-75,225) [vertex,color=brown] (72) {};
			\node at (-25,225) [vertex,color=orange] (73) {};
			\node at (25,225) [vertex] (74) {};
			\node at (75,225) [vertex] (75) {};
			\node at (125,225) [vertex,color=violet] (76) {};
			\node at (175,225) [vertex,color=teal] (77) {};
			\node at (225,225) [vertex] (78) {};

			\path
			
			(11) edge (12)
			(12) edge (13)
			(13) edge (14)
			(14) edge (15)
			(15) edge (16)
			(16) edge (17)
			(17) edge (18)
			
			(21) edge (22)
			(22) edge (23)
			(23) edge (24)
			(24) edge (25)
			(25) edge (26)
			(26) edge (27)
			(27) edge (28)
			
			(31) edge (32)
			(32) edge (33)
			(33) edge (34)
			(34) edge (35)
			(35) edge (36)
			(36) edge (37)
			(37) edge (38)
			
			(41) edge (42)
			(42) edge (43)
			(43) edge (44)
			(44) edge (45)
			(45) edge (46)
			(46) edge (47)
			(47) edge (48)
			
			(51) edge (52)
			(52) edge (53)
			(53) edge (54)
			(54) edge (55)
			(55) edge (56)
			(56) edge (57)
			(57) edge (58)
			
			(61) edge (62)
			(62) edge (63)
			(63) edge (64)
			(64) edge (65)
			(65) edge (66)
			(66) edge (67)
			(67) edge (68)
			
			(71) edge (72)
			(72) edge (73)
			(73) edge (74)
			(74) edge (75)
			(75) edge (76)
			(76) edge (77)
			(77) edge (78)

			(11) edge (21)
			(21) edge (31)
			(31) edge (41)
			(41) edge (51)
			(51) edge (61)
			(61) edge (71)
			
			(12) edge (22)
			(22) edge (32)
			(32) edge (42)
			(42) edge (52)
			(52) edge (62)
			(62) edge (72)
			
			(13) edge (23)
			(23) edge (33)
			(33) edge (43)
			(43) edge (53)
			(53) edge (63)
			(63) edge (73)
			
			(14) edge (24)
			(24) edge (34)
			(34) edge (44)
			(44) edge (54)
			(54) edge (64)
			(64) edge (74)
			
			(15) edge (25)
			(25) edge (35)
			(35) edge (45)
			(45) edge (55)
			(55) edge (65)
			(65) edge (75)
			
			(16) edge (26)
			(26) edge (36)
			(36) edge (46)
			(46) edge (56)
			(56) edge (66)
			(66) edge (76)
			
			(17) edge (27)
			(27) edge (37)
			(37) edge (47)
			(47) edge (57)
			(57) edge (67)
			(67) edge (77)
			
			(18) edge (28)
			(28) edge (38)
			(38) edge (48)
			(48) edge (58)
			(58) edge (68)
			(68) edge (78)

			;

		\end{tikzpicture}
		\caption{A partial colouring of $P_7 \cp P_8$}
		\label{fig:P_7 cp P_8}
	\end{figure}

	It would be of interest to determine exactly when the lower bound in Proposition~\ref{cartesian product bound} can be met for grids. We now investigate cylinder and torus graphs. The $\gp $-numbers of cylinders, i.e.\ Cartesian products of paths and cycles, were determined in~\cite{Cartesian}.
	\begin{theorem}[\cite{Cartesian}]\label{thm:gp of cylinder}
		If $n_1 \geq 2$ and $n_2 \geq 3$, then
		\[ \gp(P_{n_1} \cp C_{n_2}) =\begin{cases}
			3, & \text{ if } n_1 = 2, n_2 = 3,  \\
			5, & \text{ if } n_1 \geq 5, \text{ and } n_2 = 7 \text{ or } n_2 \geq 9,  \\
			4, & \text{ otherwise. }  \\
		\end{cases}\]
		
	\end{theorem}
	
	Trivially $\chi _{\gp }(P_2 \cp C_3) = 2$. It is shown in~\cite{effopendom} that if $n_1$ is odd and $4|n_2$, then $P_{n_1} \cp C_{n_2}$ is an efficient open domination graph, as is the torus graph $C_{n_1} \cp C_{n_2}$ when $4|n_1$ and $4|n_2$.  This implies the following upper bound.
	
	\begin{corollary}
		The $\gp $-chromatic number of cylinder and torus graphs $P_{n_1} \cp C_{n_2}$ ($n_1 \geq 2,n_2 \geq 3$, and $(n_1,n_2) \neq (2,3)$) and $C_{n_1} \cp C_{n_2}$ ($n_1,n_2 \geq 3$) is bounded above by $\frac{n_1n_2}{4}+O(n_1+n_2)$.
	\end{corollary}
	
	When $\gp (P_{n_1} \cp C_{n_2}) = 5$ we obtain a better upper bound.
	\begin{theorem}
		When $n_1 \geq 5$ and $n_2 = 7$ or $n_2 \geq 9$, \[ \chi _{\gp }(P_{n_1} \cp C_{n_2}) \leq \frac{n_1n_2}{5}+2n_2.\]
	\end{theorem}
	\begin{proof}
		It is shown in~\cite{Cartesian} that $\gp $-sets of $P_{n_1} \cp C_{n_2}$ are given by \[ \{ (1,0),(2,2),(3,4),(4,6),(5,1)\} , \{ (1,1),(2,4),(3,\left \lfloor \frac{n_2}{2} \right \rfloor +2),(4,0),(5,3)\} \] for $n_1 \geq 5$ and $n_2 = 7$ and $n_1 \geq 5$ and $n_2 \geq 9$ respectively. By `rotating' these sets around the cylinder, i.e. adding the same amount to each of the second coordinates, we tessellate a section of the cylinder perfectly. We can repeat this argument by adding $5t$ to the first coordinates of these sets for $t \leq \lfloor \frac{n_1}{5} \rfloor -1$. In this way at most $4n_2$ vertices are not contained in colour classes of size five. If $n_1 \equiv 0 \pmod{4}$, then we have partitioned $V(P_{n_1} \cp C_{n_s})$ into $\gp $-sets of size five. If $n_1 \equiv 2,4 \pmod{5}$, then the remaining vertices can be partitioned into rotated copies of the general position set $\{ (1,0),(2,1),(1,\left \lfloor \frac{n_2}{2}\right \rfloor ),(2,\left \lfloor \frac{n_2}{2} \right \rfloor +1)\} $.
	\end{proof}
	
	The classification of general position numbers of torus graphs was completed in~\cite{torus}. The details for small toruses $C_{n_1} \cp C_{n_2}$ are complicated, so we remark only that the largest possible general position number of a torus is seven and this is always achieved if $n_1 \geq n_2 \geq 13$. We show that for infinitely many toruses we can meet the lower bound $\frac{n_1n_2}{7}$. When $n_1 \geq n_2 \geq 49$, $7|n_1$ and $7|n_2$, the first coordinates of the vertices in the general position sets exhibited in~\cite{torus} are $\frac{in_1}{7}$ for $0 \leq i \leq 6$. Therefore, following the strategy above, we can rotate these sets in the second coordinate to cover seven rows of the torus with $\gp $-sets of order seven. By then rotating these strips in the first coordinate we obtain a tessellation of the vertices of the torus.

	\begin{corollary}
		When $s \geq t \geq 7$, \[ \chi _{\gp }(C_{7s} \cp C_{7t}) = 7st.\] 
	\end{corollary}
	
	We also mention briefly the strong product of paths, i.e. a strong grid. It is shown in~\cite{KlavzarYero} that $\gp (P_{n_1} \boxtimes P_{n_2}) = 4$, but in~\cite{Cicerone} that $\mu (P_{n_1} \boxtimes P_{n_2}) = 2(n_1+n_2)-4$, with a maximum mutual-visibility set corresponding to the `border' of the grid.
	\begin{theorem}
		For $2 \leq n_1 \leq n_2$, $\chi _{\gp }(P_{n_1} \boxtimes P_{n_2}) \leq \frac{n_1n_2}{4}+n_2$ and $\chi _{\gp }(P_{n_1} \boxtimes P_{n_2}) = \frac{n_1n_2}{4}$ when $n_1$ and $n_2$ are even. The $\mu $-chromatic number is given by $\chi _{\mu }(P_{n_1} \boxtimes P_{n_2}) = \left \lceil \frac{n_1}{2} \right \rceil $ (unless $n_1 = 2$, in which case $\chi _{\mu }(P_{n_1} \boxtimes P_{n_2}) = 2$).
	\end{theorem}
	\begin{proof}
		The vertex set of the strong grid can be divided into cliques of order four, with at most $n_1+n_2-1$ remaining vertices, which yields the bound for $\gp $-colourings. As any mutual-visibility set of the strong grid can contain at most two vertices along any diagonal, there must be at least $\left \lceil \frac{n_1}{2} \right \rceil $ colours in any $\mu $-colouring. If $n_1 \geq 3$ a $\mu $-colouring is given by the colour classes $\{ i,r+i\} \times [n_2]$ for $1 \leq i \leq \left \lfloor \frac{n_1}{2} \right \rfloor $, with an additional colour class $\{ n_1\} \times [n_2]$ if $n_1$ is odd.
	\end{proof}
	This is another example that shows that $\chi _{\gp }(G)$ can be arbitrarily large for a fixed value of $\chi _{\mu }(G)$,

	\section{Computational complexity}\label{sec:complexity}
	To study the computational complexity of finding a $\gp$-colouring of a graph that uses the smallest possible number of colours, we introduce the following decision problem.
	\begin{definition}
		{\sc GP-colouring} problem: \\
		{\sc Instance}: A graph $G=(V,E)$, a positive integer $k\leq |V|$. \\
		{\sc Question}: Is there a $\gp$-colouring of $G$ such that $\chi_{\gp }(G)\leq k$?
	\end{definition}
	The next theorem shows that the problem is hard to solve.
	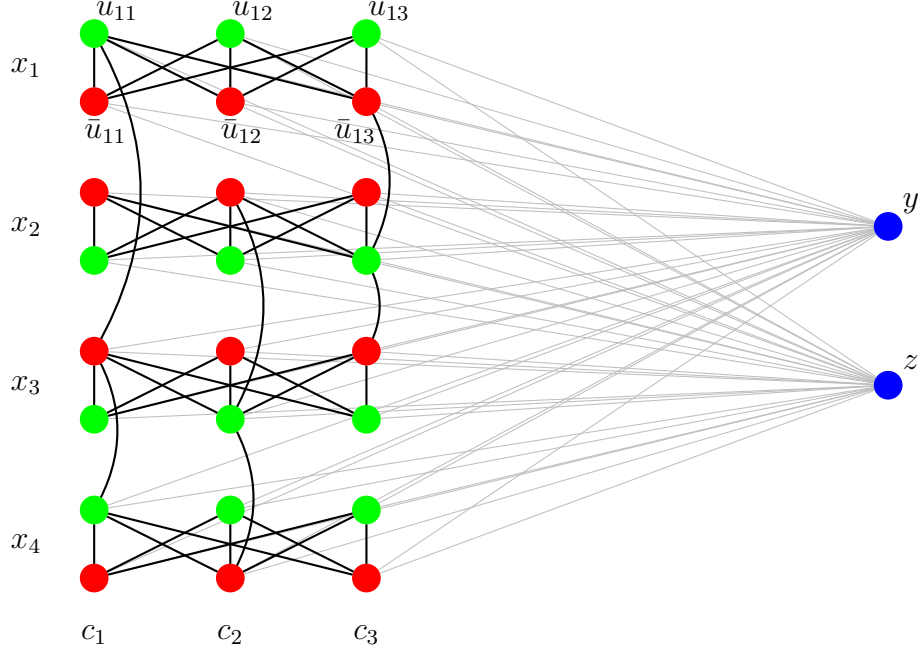
\begin{figure}[t]
		\begin{center}
			\begin{tikzpicture}[x=0.2mm,y=-0.2mm,inner sep=0.1mm,scale=1.50,
				thick,vertex/.style={circle,draw,minimum size=10,font=\tiny,fill=white},edge label/.style={fill=white}]
				\usetikzlibrary {backgrounds}
				
				%\small
				\node (x1) at (-80,0) {$x_1$};
				\node (x2) at (-80,70) {$x_2$};
				\node (x3) at (-80,140) {$x_3$};
				\node (x4) at (-80,210) {$x_4$};
				\node (y) at (310,60) {$y$};
				\node (z) at (310,130) {$z$};
				\node (u11) at (-40,-25) {$u_{11}$};
				\node (u12) at (20,-25) {$u_{12}$};
				\node (u13) at (80,-25) {$u_{13}$};
				\node (nu11) at (-45,28) {$\bar u_{11}$};
				\node (nu12) at (15,28) {$\bar u_{12}$};
				\node (nu13) at (65,28) {$\bar u_{13}$};
				\node (c1) at (-50,250) {$c_1$};
				\node (c2) at (10,250) {$c_2$};
				\node (c3) at (70,250) {$c_3$};
				\path[draw=black] %clauses
				(-50,-15) edge[bend left] (-50,125)
				(-50,125) edge[bend left] (-50,195)
				
				(10,55) edge[bend left] (10,155)
				(10,155) edge[bend left] (10,225)
				
				(70,15) edge[bend left] (70,85)
				(70,85) edge[bend left] (70,125)
				;

				\node at (300,70)  [vertex,color=blue] (y) {$y$};
				\node at (300,140) [vertex,color=blue] (z) {z};
				
				\node at (-50,-15) [vertex,color=green] (a1) {$x_1$};
				\node at (10,-15) [vertex,color=green] (a2) {$x_2$};
				\node at (70,-15) [vertex,color=green] (a3) {$x_3$};
				
				\node at (-50,15) [vertex,color=red] 
				(b1) {$y_1$};
				\node at (10,15) [vertex,color=red] (b2) {$y_2$};
				\node at (70,15) [vertex,color=red] (b3) {$y_3$};

				\path
				(a1) edge (b1)
				(a1) edge (b2)
				(a1) edge (b3)
				
				(a2) edge (b2)
				(a2) edge (b1)
				(a2) edge (b3)
				
				(a3) edge (b3)
				(a3) edge (b2)
				(a3) edge (b1)
				;
				\begin{scope}[on background layer]
					\path[draw=lightgray]
					(y) edge (a1) edge (a2) edge (a3)
					(y) edge (b1) edge (b2) edge (b3)
					(z) edge (a1) edge (a2) edge (a3)
					(z) edge (b1) edge (b2) edge (b3)
					;
				\end{scope}
				\node at (-50,55) [vertex,color=red] (a1) {$x_1$};
				\node at (10,55) [vertex,color=red] (a2) {$x_2$};
				\node at (70,55) [vertex,color=red] (a3) {$x_3$};
				
				\node at (-50,85) [vertex,color=green] 
				(b1) {$y_1$};
				\node at (10,85) [vertex,color=green] (b2) {$y_2$};
				\node at (70,85) [vertex,color=green] (b3) {$y_3$};

				\path
				(a1) edge (b1)
				(a1) edge (b2)
				(a1) edge (b3)
				
				(a2) edge (b2)
				(a2) edge (b1)
				(a2) edge (b3)
				
				(a3) edge (b3)
				(a3) edge (b2)
				(a3) edge (b1)	
				;
				\begin{scope}[on background layer]\path[draw=lightgray]
					(y) edge (a1) edge (a2) edge (a3)
					(y) edge (b1) edge (b2) edge (b3)
					(z) edge (a1) edge (a2) edge (a3)
					(z) edge (b1) edge (b2) edge (b3)
					;
				\end{scope}
				\node at (-50,125) [vertex,color=red] (a1) {$x_1$};
				\node at (10,125) [vertex,color=red] (a2) {$x_2$};
				\node at (70,125) [vertex,color=red] (a3) {$x_3$};
				
				\node at (-50,155) [vertex,color=green] 
				(b1) {$y_1$};
				\node at (10,155) [vertex,color=green] (b2) {$y_2$};
				\node at (70,155) [vertex,color=green] (b3) {$y_3$};

				\path
				(a1) edge (b1)
				(a1) edge (b2)
				(a1) edge (b3)
				
				(a2) edge (b2)
				(a2) edge (b1)
				(a2) edge (b3)
				
				(a3) edge (b3)
				(a3) edge (b2)
				(a3) edge (b1)		
				;
				\begin{scope}[on background layer]
					\path[draw=lightgray]
					(y) edge (a1) edge (a2) edge (a3)
					(y) edge (b1) edge (b2) edge (b3)
					(z) edge (a1) edge (a2) edge (a3)
					(z) edge (b1) edge (b2) edge (b3)
					;
				\end{scope}
				\node at (-50,195) [vertex,color=green] (a1) {$x_1$};
				\node at (10,195) [vertex,color=green] (a2) {$x_2$};
				\node at (70,195) [vertex,color=green] (a3) {$x_3$};
				
				\node at (-50,225) [vertex,color=red] 
				(b1) {$y_1$};
				\node at (10,225) [vertex,color=red] (b2) {$y_2$};
				\node at (70,225) [vertex,color=red] (b3) {$y_3$};

				\path
				(a1) edge (b1)
				(a1) edge (b2)
				(a1) edge (b3)
				
				(a2) edge (b2)
				(a2) edge (b1)
				(a2) edge (b3)
				
				(a3) edge (b3)
				(a3) edge (b2)
				(a3) edge (b1)		
				;
				\begin{scope}[on background layer]
					\path[draw=lightgray]
					(y) edge (a1) edge (a2) edge (a3)
					(y) edge (b1) edge (b2) edge (b3)
					(z) edge (a1) edge (a2) edge (a3)
					(z) edge (b1) edge (b2) edge (b3)
					;
				\end{scope}
				
			\end{tikzpicture}
		\end{center}
		\caption{The graph $G$ used in Theorem~\ref{theo:gpc-NP} corresponding to an instance $\Phi=(X,C)$ of {\sc NAE3-SAT} with $X=\{x_1,x_2,x_3,x_4\}$ and $C=\{c_1,c_2,c_3\}$, where $c_1=\{x_1,x_3,x_4\}$, $c_2=\{x_2,\bar x_3, \bar x_4\}$ and $c_3=\{\bar x_1,\bar x_2,x_3\}$. 
			To each variable $x_i$ is associated a $K^{(i)}$ subgraph isomorphic to $K_{3,3}$, and to each form of the variable, $x_i$ or $\bar x_i$, are associated three vertices of the $K^{(i)}$ subgraph (namely $u_{ij}$ or $\bar u_{ij}$, $j\in [3]$), that is one vertex for each clause.
			The three clauses are represented by the three $P_3$ paths with bent edges. Green and red colours represent True and False assignments to the variables given by a function $t$. In this case $t(x_1)=True$, $t(x_2)=False$, $t(x_3)=False$ and $t(x_4)=True$, and $\Phi$ is satisfied. Accordingly, the colouring of $G$ is a $\gp$-colouring (since $\diam(G)=2$ and each $P_3$ subgraph is not monochromatic) and $\chi_{\gp}(G) =3$.}
		\label{fig:NP}
	\end{figure} 
	
	\begin{theorem}\label{theo:gpc-NP}
		{\sc GP-Colouring} is NP-complete even for instances $(G,k)$ with \newline $\diam(G) = 2$ and $k=3$.  
	\end{theorem}  
	
	\begin{proof} 
		Given a colouring of $G$, it is possible to test in polynomial time whether it is a $\gp$-colouring or not. Consequently, the problem is in NP. 
		We prove that the {\sc NAE3-SAT} problem, shown to be NP-complete in~\cite{schafer-1978}, polynomially reduces to {\sc GP-Colouring}.
		\begin{quote}  
			A {\sc NAE3-SAT} (\emph{not-all-equal 3-satisfiability}) instance $\Phi$ is defined as a set $X=\{x_1,x_2,\ldots,x_p\}$ of $p$ Boolean variables and a set $C$ of $q$ clauses, each  defined as a set of three literals: every variable $x_i$ corresponds to two literals $x_i$ (the positive form) and $\bar x_i$ (the negative form). To simplify the notations we will denote by $\{\ell_1, \ell_2, \ell_3\}$ the clause with literals $\ell_i, i\in[3]$, without distinction between the orders in which they are listed. A truth assignment assigns a Boolean value ($True$ or $False$) to each variable, corresponding to a truth assignment of opposite values for the two literals $x_i$ and $\bar x_i$: $\bar x_i$ is $True$ if and only if $x_i$ is $False$. 
			
			The {\sc NAE3-SAT} problem asks whether there is a truth assignment  to the variables such that in no clause all three literals have the same truth value. We will say that such an assignment is \emph{satisfying} and the instance $\Phi $ is \emph{satisfied}.
		\end{quote}
		In what follows, we may assume that any instance of {\sc NAE3-SAT} has at least three clauses and that each clause involves three different variables. Indeed, if a clause involves a single variable and all the literals are equal, then the {\sc NAE3-SAT} problem has a No answer, otherwise any instance $\Phi$ that does not satisfy the above assumption can be transformed into an equivalent instance $\Phi'$ satisfying it. If the positive and negative form of a variable are present in a clause we can remove the clause from the instance.
		If a clause involves only two variables, and the variable appearing two times in the clauses is present in the same form, e.g. $\{l_1,l_1,l_2\}$, then the clause is logically equivalent to $\{l_1,l_2\}$ and it can be substituted with two clauses $\{l_1,l_2, a\}$ and $\{l_1,l_2, \bar a\}$, where $a$ is an additional variable. Then the {\sc NAE3-SAT} instance $\Phi$ has a Yes answer if and only if $\Phi'$ has a Yes answer.
		
		We reduce any instance of {\sc NAE3-SAT} to an instance of {\sc GP-Colouring}. Let $\Phi=(X,C)$, with $X = \{x_1,x_2,\ldots,x_p\}$
		and $C = \{c_1,c_2,\ldots,c_q\}$, be any instance of {\sc NAE3-SAT}. We must construct a graph $G = (V,E)$ and a positive integer $k \leq |V|$ such that $G$ has a $\gp$-colouring  of size $k$ or less if and only if $\Phi$ admits a satisfying truth assignment.
		
		For each variable $x_i\in X$, $i\in [p]$, there is a true-setting  subgraph $K^{(i)}=(V_i,E_i)$ of $G$ with $V_i=\{u_{ij},\bar u_{ij}~|~j\in[q]\}$ and $E_i=\{u_{ia}\bar u_{ib}~|~a,b\in[q]\}$. Then each subgraph $K^{(i)}$ is isomorphic to the complete bipartite graph $K_{q,q}$. Notice that, by Theorem~\ref{thm:multipartite}, $K^{(i)}$ can be coloured with two colours and, since $q\geq 3$, all the vertices $u_{ij}$, for $j\in[q]$, must have the same colour, whereas the other colour is reserved for the vertices $\bar u_{ij}$.
		
		For each clause $c_j \in C$, $j\in [q]$, there is subgraph $P^{(j)}=(V^j,E^j)$ of $G$ isomorphic to a $P_3$ (a path graph with three vertices). If the clause $c_j$ contains the variables $x_{i_1},x_{i_2}$, and $x_{i_3}$, then the three vertices of $P^{(j)}$ belong to $K^{(i_1)}$, $K^{(i_2)}$, and $K^{(i_3)}$. In particular, if $x_{i_1}$ appears in $c_j$ in positive (negative, resp.) form, then $u_{i_1j}$ ( $\bar u_{i_1j}$, resp.) is a vertex of $P^{(j)}$.  The same is true of the vertices corresponding to variables $x_{i_2}$ and $x_{i_3}$. The order in which the vertices appear in the path is not relevant.
		
		There are two more non-adjacent vertices in $V$, which we call $y$ and $z$. We set these two vertices to be adjacent to all the other vertices in $G = (V,E)$. Then
		
		$$ V=\bigcup_{i\in [p]} V_i \cup\{y,z\}$$
		
		$$ E=\bigcup_{i\in [p]} E_i \cup \bigcup_{j\in[q]} E^j\cup\{u_{ij}y,\bar u_{ij}y,u_{ij}z,\bar u_{ij}z~|~i\in[p],j\in[q]\}$$

		A representation of $G$ is given in Figure~\ref{fig:NP}.
		
		The construction of our instance of {\sc GP-Colouring} is completed by setting $k= 3$.
		It is easy to see how the construction can be accomplished in polynomial time and that $\diam(G)=2$. All that remains to be shown is that $\Phi$ is satisfied if and only if
		$G$ has a $\gp$-colouring of size $k$ or less.
		
		First, suppose that $t: X\rightarrow \{True,False\}$ is a satisfying truth assignment for $C$. 
		Then we colour the vertices in $G$ with three colours, namely red, green and blue, depending on $t$. The vertices $y$ and $z$ are coloured blue. If $t(x_i)$ is $True$ for a given $i\in [p]$, then all the vertices $u_{ij}$, $j\in[q]$, in subgraph $K^{(i)}$ are coloured green and all the vertices $\bar u_{ij}$ are coloured red. Conversely, if $t(x_i)$ is $False$, then all the vertices $u_{ij}$ are coloured red and all the vertices $\bar u_{ij}$ are coloured green.
		
		We now show that this colouring is a $\gp$-colouring. Clearly, the blue colour class $\{ y,z\} $ is in general position. Consider now two green vertices $u,v$; we demonstrate that any shortest $u,v$-path in $G$ does not contain another green vertex.  If $u$ and $v$ both belong to $K^{(i)}$, for the same $i\in [p]$, then they are at distance two and all the paths $u,w,v$ connecting them are such that $w$ is a red vertex in  $K^{(i)}$, or $w$ is blue, i.e.\ $w \in \{y,z\}$. If $u$ belongs to $K^{(i)}$ and $v$ belongs to $K^{(h)}$ for some $h\in[p]$ such that $h\not = i$, they could be adjacent only if they belong to a $P^{(j)}$ path for some $j\in[q]$. Instead, if they are at distance two, then all the paths $u,w,v$ connecting them are such that:
		\begin{enumerate}
			\item $w\in \{y,z\}$;
			\item $w$ is in $K^{(i)}$ or in $K^{(h)}$;
			\item $w$ is in $K^{(\ell)}$ with $\ell\not \in\{i,h\}$.
		\end{enumerate}
		In case 1) $w$ is blue and in case 2) $w$ is necessarily red, being adjacent to a green vertex in the same $K_{q,q}$ subgraph. In case 3) the subgraph induced by $u,w$ and $v$ is a $P^{(j)}$ path for some $j\in[q]$, and since $t$ is a satisfying assignment, the three vertices cannot have the same colour. Hence, since $u,v$ are coloured green, then $w$ is coloured red as in case 2). It follows that the set of green vertices is in general position.
		
		The same reasoning applies for pairs of red vertices, and so the colouring is a $\gp $-colouring with three colours.
		This concludes the first part of the proof.

		Conversely, let us suppose that there is a $\gp$-colouring of the graph $G$ with three colours. For a contradiction, let us assume that $y$ and $z$ belong to different colour classes. Consider the graph $G'=(V',E')$ induced by the vertices $V - \{y,z\}$. Two non-adjacent vertices $u,v$ in $V'$ cannot have the same colour as $y$, because $u,y,v$ is a shortest path. Then if there are vertices in $V'$ with the same colour of $y$, they must belong to the same clique of $G'$. But as $G'$ is triangle-free, there are at most two adjacent vertices in $V'$ with the same colour as $y$. The same holds for the vertex $z$. Then there are at least $|V'|- 4$ vertices in $V'$ that must belong to the same colour class, but then they are not in general position since $|V'|\geq 18$, and so a $K^{(i)}$ must be monochromatic for some $i\in[p]$. We conclude that $y,z$ are in the same colour class, say blue, and the other vertices in $G'$ belong to the green and red classes.
		
		We already noted that the vertices $u_{i,j}$ of a $K^{(i)}$ subgraph have the same colour, as well as the vertices $\bar u_{i,j}$, and that these colours are different. Then for each $i\in[p]$ we assign the value $True$ to variable $x_i$ if the vertices $u_{i,j}$ in $K^{(i)}$ are coloured green, whereas we assign the value $False$ to variable $x_i$ if the vertices $u_{i,j}$ in $K^{(i)}$ are coloured red.  
		
		Now consider any clause $c_j=\{\ell^j_1,\ell^j_2,\ell^j_3\}$, $j\in[q]$, and the corresponding subgraph $P^{(j)}$. Since $P^{(j)}$ is a shortest path in $G$, then the three vertices in $P^{(j)}$ cannot have the same colour. As a consequence, the three literals $\ell^j_1$, $\ell^j_2$, and $\ell^j_3$ corresponding to the three vertices in 
		$P^{(j)}$ cannot have the same truth assignment.
		As $c_j$ was a general clause, it follows that all the clauses are satisfied. 
	\end{proof}
	
	Consider now the following problem:
	
	\begin{definition}
		{\sc IGP-Colouring} problem: \\
		{\sc Instance}: A graph $G=(V,E)$, a positive integer $k\leq |V|$. \\
		{\sc Question}: Is there a $\igp$-colouring of $G$ such that $\chi_{\igp }(G)\leq k$?
	\end{definition}
	
	We have the following result:
	
	\begin{theorem}\label{theo:igpc-NP}
		{\sc IGP-Colouring} is NP-complete even for instances $(G,k)$ with \newline $\diam(G)\leq 3$.  
	\end{theorem} 
	\begin{proof}
		By Theorem~\ref{thm:diam three}, we know that  
		$\chi_{\igp }(G) = \chi(G)$ for graphs with diameter at most three. This implies that solving the {\sc IGP-Colouring} problem is NP-complete as a consequence of the results in~\cite{mertzios-2016}.
	\end{proof}

	%%%%%%%%%%%%%%%%%%%%%%%%%%%%%%%%%%%%%%%%%%
	\section{Concluding remarks}\label{sec:conclusion}
	%%%%%%%%%%%%%%%%%%%%%%%%%%%%%%%%%%%%%%%%%%
	We conclude with some open problems suggested by our discussion. In Section~\ref{sec:bounds} we investigated the largest sizes of graphs with a given position colouring number. We determined this exactly for $\chi _{\gp _i}$ and $\chi _{\mono _i}$ and found the correct asymptotic order for $\chi _{\gp }$ and $\chi _{\mono }$. It would be of interest to find exact expressions for the latter colouring numbers and investigate the analogous question for $\chi _{\mu }$.
	
	\begin{problem}
		What is the largest size of a graph with a given value of $\chi _{\pi }(G)$, when $\pi (G)$ is any one of $\gp (G)$, $\mono (G)$, $\mu (G)$ or $\mu _i(G)$?
	\end{problem}
	
	In some cases, our optimal position colourings were obtained by partitioning the vertex set into maximal position sets, or by `nearly tessellating' the vertex set with such sets. This suggests the following packing problem. 
	\begin{problem}
		For a given graph $G$, what is the largest number of vertex-disjoint maximal position sets contained in $G$? For which graphs $G$ can $V(G)$ be partitioned into maximal position sets?
	\end{problem}
	
	Finally we suggest investigating Nordhaus-Gaddum relations for position colourings. It follows from Lemma~\ref{lem:chi inequality} together with the result of~\cite{NordGad} that $\chi _{\pi _i}(G)\chi _{\pi _i}(\overline{G}) \geq \chi (G)\chi (\overline{G}) \geq n$ and $\chi _{\pi _i}(G) +\chi _{\pi _i}(\overline {G}) \geq \chi(G)+\chi(\overline{G}) \geq 2\sqrt{n}$ for any graph $G$ whenever $\pi (G)$ is $\gp (G)$ or $\mono (G)$. Moreover, both of these bounds are tight, as evidenced respectively by a complete graph and the complete $r$-partite graph with every part of order $r$. By Lemma~\ref{lem:clique cover bound} we have $\chi _{\pi }(G)+\chi _{\pi }(\overline{G}) \leq \theta (G)+\theta (\overline{G}) \leq n+1$; up to order nine the only example of equality is $C_5$ for $\chi _{\mono }$~\cite{erskine}. We conjecture that this can be strengthened as follows.
	
	\begin{conjecture}\label{conj:NGrelation}
		For any graph $G$, if $\pi (G)$ is $\gp (G)$ or $\mono (G)$, then \[ \chi _{\pi _i}(G) +\chi _{\pi _i}(\overline{G}) \leq n+1 \text{ and } \chi _{\pi _i}(G)\chi _{\pi _i}(\overline {G}) \leq \left ( \frac{n+1}{2} \right )^2.\] 
	\end{conjecture}
	If Conjecture~\ref{conj:NGrelation} is true, then it is tight (the bounds are met respectively by $K_n$ and $K_{r+1}$ with $r$ leaves attached to a vertex). It is true for all graphs with up to ten vertices, as shown by a computer search by Erskine~\cite{erskine} (the search uses the geng program in nauty, computing all maximum $\gp $- and $\mono $-sets in the graph and then for $k = 1,2,3,\dots $ taking the union of all possible combinations of $k$ $\gp $- or $\mono $-sets until finding one that covers the whole vertex set).
	
	%%%%%%%%%%%%%%%%%%%%%%%%%%%%%%%%%%%%%%%%%%
	\section*{Acknowledgements}
	%%%%%%%%%%%%%%%%%%%%%%%%%%%%%%%%%%%%%%%%%%
	
	James Tuite gratefully acknowledges funding support from London Mathematical Society grant [ECF-2021-27] and EPSRC grant [EP/W522338/1]. Haritha S acknowledges University of Kerala for providing JRF. Gabriele Di Stefano has been supported in part by the Italian National Group for Scientific Computation (GNCS-INdAM). The authors thank Sandi Klav\v{z}ar, Grahame Erskine and Terry Griggs for valuable discussion of this material, as well as the three anonymous reviewers for their helpful comments.
	
	%%%%%%%%%%%%%%%%%%%


\begin{thebibliography}{99}
		\bibliographystyle{plain}
		
		\bibitem{BreSamYer} B.~Bre\v{s}ar, B.~Samadi, I.~G.~Yero, Injective coloring of graphs revisited, Discrete Math.\ 346 (5) (2023) 113348.
		
		\bibitem{ullas-2016} 
		U.~Chandran S.V., G.~Jaya~Parthasarathy, 
		The geodesic irredundant sets in graphs, 
		Int.\ J.\ Math.\ Combin.\ 4 (2016) 135--143.
		
		\bibitem{Cicerone} S.~Cicerone, G.~Di Stefano, S.~Klav\v{z}ar, I.~G.~Yero, Mutual-visibility in strong products of graphs via total mutual-visibility, Discrete Appl.\ Math.\ 358 (2024) 136--146.
		
		\bibitem{survey}
		U.~Chandran S.V., S.~Klav\v{z}ar, J.~Tuite, 
		The general position problem: a survey.
		\url{arXiv:2501.19385} (2025).
		
		\bibitem{CH} E.~J.~Cockayne, S.~T.~Hedetniemi, Towards a theory of domination in graphs, Networks 7 (3) (1977) 247--261.
		
		\bibitem{triplesystems} C.~J.~Colbourn, A.~Rosa, Triple Systems, (1999) Oxford University Press.
		
		\bibitem{DiStefano} G.~Di Stefano, Mutual visibility in graphs, Appl.\ Math.\ Comput.\ 419 (2022) 126850.
		
		\bibitem{lowergp} G.~Di Stefano, S.~Klav\v{z}ar, A.~Krishnakumar, J.~Tuite, I.~G.~Yero, Lower general position sets in graphs, Discuss.\ Math.\ Graph Theory 45 (2) (2025) 509--531.
		
		\bibitem{dudeney-1917}
		H.~E.~Dudeney, 
		Amusements in Mathematics, 
		Nelson, Edinburgh, 1917. 
		% p. 94. 222. 
		
		\bibitem{erskine} G.~Erskine, private communication (2021).
		
		\bibitem{FolKalMcMPelWon} M.~Follett, K.~Kalail, E.~McMahon, C.~Pelland, R.~Won, Partitions of $AG(4,3)$ into maximal caps, Discrete Math.\ 337 (2014) 1--8.
		
		\bibitem{Gim}
		J.~G.~Gimbel, The Chromatic and Cochromatic Number of a Graph, Western Michigan University (1984).
		
		\bibitem{ghorbani-2019}
		M.~Ghorbani, S.~Klav\v{z}ar, H.~R.~Maimani, M.~Momeni, F.~Rahimi-Mahid, G.~Rus,
		The general position problem on Kneser graphs and on some graph operations,
		Discuss.\ Math.\ Graph Theory 41 (4) (2019) 1199--1213.  
		
		\bibitem{Grav} S.~Gravier, Total domination number of grid graphs, Discrete Appl.\ Math.\ 121 (1-3) (2002) 119--128.
		
		\bibitem{HahnKratSirSot} G.~Hahn, J.~Kratochv\'il, J.~\v{S}ir\'{a}\v{n}, D.~Sotteau, On the injective chromatic number of graphs, Discrete Math.\ 256 (2002) 179--192.
		
		\bibitem{HenYeo} M.~A.~Henning, A.~Yeo, Total Domination in Graphs, New York: Springer (2013) 10--11.
		
		\bibitem{Cartesian} S.~Klav\v{z}ar, B.~Patk\'{o}s, G.~Rus, I.~G.~Yero, On general position sets in Cartesian products, Results Math.\ 76 (3) (2021) 123.
		
		\bibitem{mvcoloring} S.~Klav\v{z}ar, D.~Kuziak, J.~C.~V.~Tripodoro, I.~G.~Yero, Coloring the vertices of a graph with mutual-visibility property, (2024) preprint arxiv.org/abs/2408.03132 
		
		\bibitem{KlavzarYero} S.~Klav\v{z}ar, I.~G.~Yero, The general position problem and strong resolving graphs, Open Math.\ 17 (1) (2019) 1126--1135.
		
		\bibitem{Klob} A.~Klobu\v{c}ar, Total domination numbers of Cartesian products, Math.\ Commun.\ 9 (1) (2004) 35--44.
		
		\bibitem{torus} D.~Kor\v{z}e, A.~Vesel, General position sets in two families of Cartesian product graphs, Mediterr.\ J.\ Math.\ 20 (4) (2023) 203.
		
		\bibitem{KraKra} F.~Kramer, H.~Kramer, Ein F\"arbungsproblem der Knotenpunkte eines Graphen bez\"uglich der Distanz $p$, Rev.\ Roumaine Math.\ Pures Appl.\ 14 (1969) 1031--1038. 
		
		\bibitem{KraKra2} F.~Kramer, H.~Kramer, Un probleme de coloration des sommets d'un graphe, C.R.\ Acad.\ Sci.\ Paris A 268 (1969) 46--48.
		
		\bibitem{effopendom} D.~Kuziak, I.~Peterin, I.G.~Yero, Efficient open domination in graph products, Discrete Math.\ Theor.\ Comput.\ Science 16 (1) (2014) 105--120.
		
		\bibitem{Lovasz} L.~Lov\'{a}sz, Kneser's conjecture, chromatic number, and homotopy, J.\ Comb.\ Theory Ser.\ A 25 (3) (1978) 319--324.
		
		\bibitem{manuel-2018a} 
		P.~Manuel, S.~Klav\v{z}ar, 
		A general position problem in graph theory, 
		Bull.\ Aust.\ Math.\ Soc.\ 98 (2018) 177--187. 
		
		\bibitem{mertzios-2016}
		G.~B.~Mertzios, P.~G.~Spirakis,
		Algorithms and almost tight results for 3-colorability of small diameter graphs,
		Algorithmica 74 (1) (2016) 385--414. 
		
		\bibitem{NordGad} E.~A.~Nordhaus, J.~Gaddum,
		On complementary graphs,
		Amer.\ Math.\ Monthly 63 (1956) 175--177.    
		
		\bibitem{Kirkmansystems} D.~K.~Ray-Chaudhuri, R.~M.~Wilson, Solution of Kirkman's schoolgirl problem, Combinatorics, Proc. Sympos.\ Pure Math.\, Univ.\ California, Los Angeles, Calif. (19) (1971) 187--203.
		
		\bibitem{schafer-1978} 
		T.~J.~Schaefer,  
		The complexity of satisfiability problems, Proceedings of the 10th Annual ACM Symposium on Theory of Computing. San Diego, California, (1978) 216--226.
		
		\bibitem{ThoCha} E.~J.~Thomas, U.~Chandran S.~V., On independent position sets in graphs, Proyecciones (Antofagasta) 40 (2) (2021) 385--398.
		
		\bibitem{ThoChaTui} E.~J.~Thomas, U.~Chandran S.~V., J.~Tuite, G.~Di Stefano, On monophonic position sets in graphs, Discrete Appl.\ Math.\ 354 (2024) 72--82.
		
		\bibitem{extremal position} J.~Tuite, E.~J.~Thomas U.~Chandran S.~V., On some extremal position problems for graphs, Ars Math.\ Contemp.\ 25 (2025) P\#1.09.
		
		\bibitem{wood} D.~R.~Wood, A note on colouring the plane grid, Geombinatorics 13 (4) (2004) 193--196.
		
		
	\end{thebibliography}
\end{document}